\documentclass{siamart1116}
\usepackage{subcaption}
\usepackage{algpseudocode}
\usepackage{algorithm}
\usepackage{amsmath}
\usepackage{amssymb}
\usepackage{scrlayer-scrpage}
\title{Localized inverse factorization%
\thanks{\today
\funding{This work was supported by the Swedish research council under grant 621-2012-3861 and the Swedish national strategic e-science research program (eSSENCE).}}}
\author{Emanuel H. Rubensson%
\thanks{Division of Scientific Computing, Department of Information Technology, Uppsala University, Box 337, SE-751 05 Uppsala, Sweden (\email{emanuel.rubensson@it.uu.se}, \email{anton.artemov@it.uu.se}, \email{anastasia.kruchinina@it.uu.se}, \email{elias.rudberg@it.uu.se}).}
\and
Anton G. Artemov%
\footnotemark[2]
\and
Anastasia Kruchinina%
\footnotemark[2]
\and
Elias Rudberg%
\footnotemark[2]
}
\headers{Localized inverse factorization}{E. H. Rubensson, A. G. Artemov, A. Kruchinina, and E. Rudberg}
\begin{document}
\maketitle
\begin{abstract}
We propose a localized divide and conquer algorithm for inverse factorization $S^{-1} = ZZ^*$ of Hermitian positive definite matrices $S$ with localized structure, e.g.\ exponential decay with respect to some given distance function on the index set of $S$. The algorithm is a reformulation of recursive inverse factorization [J. Chem. Phys., 128 (2008), 104105] but makes use of localized operations only. At each level of recursion, the problem is cut into two subproblems and their solutions are combined using iterative refinement [Phys. Rev. B, 70 (2004), 193102] to give a solution to the original problem. The two subproblems can be solved in parallel without any communication and, using the localized formulation, the cost of combining their results is proportional to the cut size, defined by the binary partition of the index set. This means that for cut sizes increasing as $o(n)$ with system size $n$ the cost of combining the two subproblems is negligible compared to the overall cost for sufficiently large systems.

We also present an alternative derivation of iterative refinement based on a sign matrix formulation, analyze the stability, and propose a parameterless stopping criterion. We present bounds for the initial factorization error and the number of iterations in terms of the condition number of $S$ when the starting guess is given by the solution of the two subproblems in the binary recursion. These bounds are used in theoretical results for the decay properties of the involved matrices.

The localization properties of our algorithms are demonstrated for matrices corresponding to nearest neighbor overlap on one-, two-, and three-dimensional lattices as well as basis set overlap matrices generated using the Hartree--Fock and Kohn--Sham density functional theory electronic structure program Ergo [SoftwareX, 7 (2018), 107].

\end{abstract}
\section{Introduction}
A standard approach to the solution of generalized eigenproblems
\begin{equation} \label{eq:gen_eig_problem}
  Fc = \lambda Sc,
\end{equation}
where $F$ and $S$ are Hermitian and $S$ is positive definite, makes
use of an inverse factor of $S$ to transform the eigenvalue problem to
standard form~\cite{elpa_2014, martin_wilkinson_1968}.
If we have that
\begin{equation}
  S^{-1} = ZZ^*,
\end{equation}
then~\eqref{eq:gen_eig_problem} implies
\begin{equation} \label{eq:std_eig_problem}
  F_\perp x = \lambda x
  \quad \mathrm{with} \quad
  F_\perp = Z^*FZ  
  \quad \mathrm{and} \quad
  Zx = c.
\end{equation}
Inverse factors may also be used in the solution of linear systems,
either as part of a direct solver or as approximate inverse
preconditioners for iterative solvers.
We are here motivated by the solution of the self-consistent field
equations appearing in a number of electronic structure models such as
Hartree--Fock and Kohn--Sham density functional theory. In this
context,
$F$ is the Fock/Kohn--Sham matrix and 
the density matrix is given by a subset of the eigenvectors
of~\eqref{eq:gen_eig_problem} as
\begin{equation} \label{eq:dens_mat}
  D = \sum_{i=1}^{\mathrm{nocc}} c_ic_i^*
\end{equation}
where nocc is the number of occupied electron orbitals and the
eigenvalues are ordered in ascending order, $\lambda_1 \leq \lambda_2
\leq \dots \leq \lambda_{\mathrm{nocc}} < \lambda_{\mathrm{nocc}+1} \leq
\dots \leq \lambda_{n-1} \leq \lambda_n$. The corresponding density
matrix for the standard problem~\eqref{eq:std_eig_problem} is given by
\begin{equation}
  D_\perp = \sum_{i=1}^{\mathrm{nocc}} x_ix_i^*
\end{equation}
which is related to~\eqref{eq:dens_mat} by $D = Z D_\perp Z^*$.  In
self-consistent field calculations, $S$ is the basis set overlap
matrix, in other contexts referred to as the mass or Gram matrix.
Some of the most efficient methods to compute the density matrix
assume the eigenproblem is on standard
form~\cite{JaehoonYousung_2016, niklasson_2002, nonmono,
  assessment_2011, Suryanarayana_2013}. Unless $S=I$ (corresponding to
an orthogonal basis) the computation of the density matrix then
consists of three steps~\cite{book-chapter-rubensson}:
\begin{algorithmic}[1]
  \State Congruence transformation $F_\perp = Z^*FZ$
  \State Density matrix method to compute $D_\perp$ from $F_\perp$
  \State Congruence transformation $D = Z D_\perp Z^*$
\end{algorithmic}
A simple and efficient method for the second step sees the problem as
a matrix function $D_\perp = \theta(\mu I-F_\perp)$ where $\theta$ is
the Heaviside function and $\mu \in ]\lambda_{\textrm{nocc}},
  \lambda_{\textrm{nocc}+1}[$
and uses a polynomial expansion in $F_\perp$ to compute $D_\perp$.
The computational kernel is matrix-matrix multiplication for which
implementations achieving good performance usually exist regardless
of computational platform.  For systems with nonvanishing eigenvalue
gap at $\mu$ and local basis sets, the computational complexity with
respect to system size can be reduced to linear, $O(n)$, with control
of the forward error~\cite{accPuri}.  The foundation for linear
scaling methods has been discussed extensively in the literature, see
e.g.~\cite{benzi-decay, cloizeaux_1964, kohn_1959,
  kohn-nearsightedness_1996}. Of particular importance are the decay
properties of the density matrix~\cite{benzi-decay} that are the basis
of sparse approximations where only $O(n)$ matrix entries are
stored~\cite{Li1993-cutoff, sparsity-threshold, bringing_about}.  Here
we note that to take advantage of the decay properties in the
three-step approach outlined above, an inverse factor $Z$ that allows
for a sparse approximation with $O(n)$ entries must be used.

There is an infinite number of matrices $Z$ that are inverse factors
of $S$, i.e.~that fulfill $S^{-1}=ZZ^*$. The ones most commonly used
are the inverse square root (L{\"o}wdin)~\cite{Branislav_2007, Lowdin1956,
  VandeVondele_2012} and inverse Cholesky
factors~\cite{challacombe_1999, hierarchic_2007}, both of which have
the important property of decay of matrix element magnitude with
atomic separation~\cite{benzi-decay}.  The cost of standard methods
for their computation in general scales cubically with system size. A
linearly scaling alternative is based on iterative
refinement~\cite{Niklasson2004} which, given a starting guess $Z_0$,
produces a sequence of matrices $Z_1,Z_2,\dots$ with convergence to an
inverse factor of $S$ if the initial factorization error
\begin{equation}\label{eq:iter_refine_conv_crit}
\|I - Z_0^*SZ_0\|_2 < 1.
\end{equation}
 The iterative refinement converges to the inverse square root if
 $Z_0$ commutes with $S$ and to some other inverse factor otherwise.
 The starting guess is often set to the identity matrix, scaled so
 that~\eqref{eq:iter_refine_conv_crit} is
 fulfilled~\cite{Branislav_2007, VandeVondele_2012}.  Being based on
 matrix-matrix multiplication the iterative refinement approach has
 advantages similar to those of the recursive density matrix
 expansions for the second step in the three-step approach described
 above.

Recursive inverse factorization, considered in the present work, also
makes use of iterative refinement but does so in a hierarchical
manner~\cite{rubenssonBockHolmstromNiklasson}. A binary partition of
the index set $\{1,\dots,n\}$ of $S$ is applied corresponding to a binary principal
submatrix decomposition
\begin{equation}
  S = \begin{bmatrix}A & B \\B^* & C\end{bmatrix},
\end{equation}
inverse factors $Z_A$ and $Z_C$ of $A$ and $C$, respectively, are
computed and
\begin{equation}
  Z_0 = \begin{bmatrix}Z_A & 0 \\ 0 & Z_C\end{bmatrix}
\end{equation}
is used as starting guess for iterative refinement.  This binary
partitioning is applied recursively which gives the recursive inverse
factorization method.
The initial factorization error depends on the binary partition of the
index set and one may thus attempt to partition the index set so that
the initial error becomes as small as possible.  It was shown
in~\cite{rubenssonBockHolmstromNiklasson} that \emph{any} binary partition
or cut of the index set in two pieces
leads to~\eqref{eq:iter_refine_conv_crit} being fulfilled.
The index set may for example correspond to the vertices of a graph or
centers of atom-centered basis functions.  In this article, the terms
vertex and index are used interchangeably.

In the present work, we further analyze the recursive inverse
factorization method and propose a localized version. This variant
exhibits more localized computations and is even more amenable to
parallelization.
Under certain assumptions, including $S$ being
localized with respect to some given distance
function on its index set and provided that matrix entries with
magnitude below some fixed threshold value are discarded, we show that
using this localized inverse factorization the workload for the
iterative refinement used to glue together $Z_A$ and $Z_C$ is
proportional to the cut size.
Thus, provided that a principal submatrix cut can be done so that only
a small number of vertices, e.g. $o(n)$, are close to the cut, the
cost of the iterative refinement is negligible compared to the inverse
factorizations of $A$ and $C$ for sufficiently large systems,
e.g.\ $o(n)$ versus $O(n)$.  In case $S$ is the overlap matrix for a
local atom-centered basis set, the distance function may, with our
formulation of localization, be taken as the Euclidean distance
between basis function centers. In this case it is usually straightforward to
make a binary division such that only $o(n)$ vertices are close to the
cut. The two subproblems to compute inverse factors of $A$ and $C$ are
completely disconnected and thus embarrassingly parallel.

In Section~\ref{sec:iter_from_sign} we propose a sign matrix
formulation for inverse factorization. This formulation is in
Section~\ref{sec:sign_matrix} used in an alternative derivation of
iterative refinement, making the relation to sign matrix and density
matrix expansion methods evident.  The stability of both regular and
localized iterative refinement is considered and new stopping criteria
are proposed. In Section~\ref{sec:binary_decomp} we introduce the
binary principal submatrix decomposition and regular and localized
iterative refinement algorithms including starting guesses given by
the binary principal submatrix decomposition. We also derive
convergence results, giving a bound for the number of iterative
refinement iterations.  In Section~\ref{sec:localization} we introduce
the notion of exponential decay with respect to distance between
vertices and exponential decay away from the cut. We derive
localization results for the matrices occurring in regular and
localized iterative refinement. In Section~\ref{sec:recursive} we
present the full recursive and localized inverse factorization
algorithms. In Section~\ref{sec:numerexp} we present numerical
experiments to demonstrate the localization properties of the
recursive and localized inverse factorization algorithms. We end the
article with concluding remarks in Section~\ref{sec:concl}.

\section{Inverse factorization from a sign matrix formulation}\label{sec:iter_from_sign}
We present here a sign matrix formulation that we will use to derive
methods for the iterative refinement of inverse factors. 

\noindent
\begin{theorem}\label{thm:invfact_from_sign}
  Let $S$ be a Hermitian positive definite matrix and assume that $Q$
  is a nonsingular matrix. Then,
  \begin{equation} \label{eq:sign_relation}
    \mathrm{sign}\left(\begin{bmatrix}0 & Q^*S \\ Q & 0\end{bmatrix}\right) = 
      \begin{bmatrix} 0 & Z^*S \\ Z & 0 \end{bmatrix}
  \end{equation}
where $Z = Q (Q^*S Q)^{-1/2}$ and $ZZ^* = S^{-1}$.
\end{theorem}
\begin{proof}
  Let 
  \begin{equation}
    X = \begin{bmatrix}0 & Q^*S \\ Q & 0\end{bmatrix}.
  \end{equation}
  Since congruence transformation preserves positive definiteness, see
  e.g.~\cite[Theorem~4.5.8]{book-matrix-analysis}, $Q^*SQ$ is
  positive definite. The matrices $QQ^*S$ and $Q^*SQ$  have the same eigenvalues
  since $Q^*S Q v = \lambda v$ implies $Q Q^*S u = \lambda u$ with $u=Qv$.
  Therefore,
  \begin{equation}
    X^2 = \begin{bmatrix} Q^*S Q & 0 \\ 0 & QQ^*S \end{bmatrix}
  \end{equation}
  is positive definite, $X$ has no eigenvalues on the imaginary axis,
  and sign$(X)$ is defined. Eq.~\eqref{eq:sign_relation} with $Z = Q
  (Q^*S Q)^{-1/2}$ follows directly from $\textrm{sign}(X) =
  X(X^2)^{-1/2}$ and $\textrm{sign}(X) = (X^2)^{-1/2}X$ and $ZZ^* = S^{-1}$ follows from $(\mathrm{sign}(X))^2 = I$.
\end{proof}

The special case of Theorem~\ref{thm:invfact_from_sign} with $Q=I$
leading to $Z=S^{-1/2}$ was shown by Higham~\cite{Higham1997}. 
This special case already shows that methods for the matrix sign
function can be used to compute the square root together with its
inverse. Theorem~\ref{thm:invfact_from_sign} can be used to reduce the
computational cost of the matrix sign function evaluation if an
approximate inverse factor is available or can be cheaply
obtained. This approximate inverse could be such that the condition
number of the problem is reduced and/or such that only a local portion
of the inverse factor needs to be updated. In general, $Z$ will not be
the inverse square root. However, if $Q$ is Hermitian and commutes
with $S$, then $Q$ and $S$
are simultaneously diagonalizable, i.e. $Q =
V\Lambda_Q V^*$ and $S = V\Lambda_S V^*$ with unitary $V$ and diagonal $\Lambda_Q, \Lambda_S$, and
\begin{align}
  Z & = Q (Q^*S Q)^{-1/2} 
   = Q ( V\Lambda_Q V^* V\Lambda_S V^* V\Lambda_Q V^* )^{-1/2} \\
  & = Q ( V\Lambda_Q^{-1} \Lambda_S^{-1/2} V^* ) 
   = Q Q^{-1}S^{-1/2} 
   = S^{-1/2}.
\end{align}
Theorem~\ref{thm:invfact_from_sign} is closely related to and can be
shown using~\cite[Lemma~4.3]{HighamMackeyMackeyTisseur_2005}. We note
that $Z$ may be computed using some method for the inverse square root
applied to $Q^*SQ$ followed by multiplication with $Q$ from the
left. We note also that the eigenvalues of $X$ are real and given as
positive-negative pairs $\lambda_1,-\lambda_1,\lambda_2,-\lambda_2,
\dots$.

\section{Iterative refinement from sign matrix methods}
\label{sec:sign_matrix}
Based on the sign matrix formulation above we provide here an
alternative derivation of the iterative refinement method
of~\cite{Niklasson2004}. The second order refinement can be derived
from the Newton--Schulz sign function iteration~\cite{book-higham}
\begin{equation} \label{eq:newton_schulz}
  X_{i+1} = \frac{1}{2}(3X_i-X_i^3)
\end{equation}
with
\begin{equation} \label{eq:single}
  X_i = \begin{bmatrix} 0 & Z_i^*S \\ Z_i & 0 \end{bmatrix}.
\end{equation}
This gives the iteration
\begin{align}
  \begin{bmatrix} 0 & Z_{i+1}^*S \\ Z_{i+1} & 0 \end{bmatrix} & =
  \frac{1}{2}\left(3 \begin{bmatrix} 0 & Z_i^*S \\ Z_i & 0 \end{bmatrix}
  - \begin{bmatrix} 0 & Z_i^*S Z_i Z_i^*S \\ Z_iZ_i^*SZ_i & 0 \end{bmatrix}\right).
\end{align}
The structure of $X_i,\ i=0,1,\dots$ is preserved and therefore only a single
channel is needed for the iteration:
\begin{equation} \label{eq:iter_refine_m2} 
  Z_{i+1} = \frac{1}{2}(3Z_i-Z_i Z_i^*SZ_i),  \quad Z_0 =  Q.
\end{equation}
Higher order polynomial iterations can be derived using the condition
that the polynomial has fixed points and a number of vanishing
derivatives at $-1$ and $1$. These can be written as
\begin{equation}\label{eq:sign_higher_order}
  X_{i+1} = X_i \sum_{k=0}^m b_k(I-X_i^2)^k
\end{equation}
where
\begin{equation}
  b_0 = 1, \quad b_k = \frac{2k-1}{2k}b_{k-1}, \quad k=1,\dots, m
\end{equation}
and $m\geq 1$.
The sign matrix residual is in each iteration reduced as
\begin{equation}
  I-X_{i+1}^2 = \sum_{k=m+1}^{2m+1} c_k(I-X_i^2)^k
\end{equation}
where 
\begin{equation}
  c_k = \sum_{j=0}^{k-m-1} 2b_jb_{k-j} - \sum_{j=0}^{k-m-2} 2b_jb_{k-j-1}.
\end{equation}
This leads to the iterative refinement 
\begin{equation}\label{eq:iter_refine}
  Z_{i+1} = Z_i \sum_{k=0}^m b_k\delta_i^k
\end{equation}
with a reduction of the error
\begin{equation}
  \delta_{i+1} = \sum_{k=m+1}^{2m+1} c_k\delta_i^k
\end{equation}
where
\begin{equation} \label{eq:delta_i}
  \delta_i = I-Z_i^*SZ_i
\end{equation}
is the factorization error in iteration $i$.  Note that
\eqref{eq:iter_refine_m2} is the special case of
\eqref{eq:iter_refine} with $m=1$.  We have that $\sum_{k=m+1}^{2m+1}
|c_k| = 1$ and therefore the iterative refinement converges if
$\|\delta_0\|_2<1$.  This also means that
\begin{equation} \label{eq:error_reduction}
\|\delta_{i+1}\|_2 <
\|\delta_{i}\|_2^{m+1}
\end{equation}
and that an accuracy $\|\delta_{k}\|_2 <
\varepsilon$ is reached within
\begin{equation}\label{eq:no-of-iters}
  k = \left\lceil\frac{\log{\left(\frac{\log \varepsilon}{\log \|\delta_0\|_2}\right)}}{\log{(m+1)}}\right\rceil
\end{equation}
iterations.

An alternative to \eqref{eq:iter_refine_m2} is given by applying the
Newton--Schulz iteration with
\begin{equation}\label{eq:dual_channel_Xi}
  X_i = \begin{bmatrix} 0 & Y_i \\ Z_i & 0 \end{bmatrix}
\end{equation}
which gives a coupled or dual channel iteration
\begin{equation}\label{eq:dual_channel}
  \left\{
    \begin{array}{r@{\;}c@{\;}lr@{\;}c@{\;}l}
      Y_{i+1} & = & \frac{1}{2}(3Y_i - Y_iZ_iY_i), & \quad Y_0 & = & Q^*S, \\
      Z_{i+1} & = & \frac{1}{2}(3Z_i - Z_iY_iZ_i), & \quad Z_0 & = & Q.
    \end{array}
  \right.
\end{equation}
This iteration has previously been considered with $Q = I$.  Since $I$
commutes with $S$, $Y_i$ and $Z_i$ then converges to the matrix square
root and its inverse, respectively, as explained in the previous
section. The corresponding higher order iterations can be obtained by
inserting~\eqref{eq:dual_channel_Xi}
in~\eqref{eq:sign_higher_order}. This alternative is not further
pursued in this work.
We note that in the present context where the inverse factor does not
have to be the inverse square root a drawback
of~\eqref{eq:dual_channel} compared to~\eqref{eq:iter_refine_m2} is
that any accuracy lost during the iterations cannot be recovered.
The polynomials in~\eqref{eq:sign_higher_order} can be seen as special
cases of the Pad{\'e} recursions for the matrix sign
function~\cite{kenney_laub_1991}. Other Pad{\'e} recursions may also
be used in the present context but are not further considered in this
article.
We note also that the Newton-Schulz iteration
in~\eqref{eq:newton_schulz} is equivalent to the McWeeny
polynomial~\cite{McWeeny1956} that has been frequently used in
computations of the density matrix. Similarly, the higher order
polynomials in \eqref{eq:sign_higher_order} correspond to the higher
order Holas polynomials~\cite{Holas2001}.
Other alternatives not further pursued in this work includes the use
of nonsymmetric polynomials in the expansion as in for example the SP2
algorithm for the density matrix~\cite{niklasson_2002} and the use of
scaling techniques to reduce the number of
iterations~\cite{KenneyLaub1992, nonmono}.

\subsection{Local refinement}\label{sec:local_refinement}
If only a local correction to the approximate inverse factor is needed
the cost of the iterative refinement can be reduced.  Assuming that
$\delta_i$ has $o(n)$ non-negligible elements its computation
using~\eqref{eq:delta_i} will for large systems involve the
computation of many matrix elements that are negligibly small.  The
factorization error can be written
\begin{align}
  \delta_i = \delta_0 - Z_i^*S(Z_i-Z_0) - (Z_i-Z_0)^*SZ_0,
\end{align}
as was noted in~\cite{rubenssonBockHolmstromNiklasson},
or as an update to the previous step factorization error
\begin{align}
  \delta_{i} = \delta_{i-1} - Z_{i}^*S(Z_{i}-Z_{i-1}) - (Z_{i}-Z_{i-1})^*SZ_{i-1}
\end{align}
which gives a dual channel iteration
\begin{equation}\label{eq:local_iter}
  \left\{
    \begin{array}{r@{\;}c@{\;}lr@{\;}c@{\;}l}
      Z_{i+1} & = & Z_i \sum_{k=0}^m b_k\delta_i^k, \\ 
      \delta_{i+1} & = & \delta_i - Z_{i+1}^*S(Z_{i+1}-Z_i) - (Z_{i+1}-Z_i)^*SZ_i. 
    \end{array}
  \right.
\end{equation}
In practice, the iteration is evaluated as
\begin{equation}\label{eq:local_iter_practice}
  \left\{
  \begin{array}{r@{\;}c@{\;}lr@{\;}c@{\;}l}
    M_i & = & Z_i (\sum_{k=1}^m b_k\delta_i^k), \\
    Z_{i+1} & = & Z_i + M_i, \\
    \delta_{i+1} & = & \delta_i - Z_{i+1}^*(SM_i) - (M_i^*S)Z_i, 
    \end{array}
  \right.
\end{equation}
where the soft brackets indicate the order of evaluation. Iteration
\eqref{eq:local_iter_practice} is a key component of the localized
inverse factorization proposed in this work. Its localization
properties will be discussed later in this article.

\subsection{Stability}\label{sec:stability}
If $Z_0$ were assumed to be Hermitian and to commute with $S$ and
exact arithmetics were used the order of the matrices in the matrix
products of~\eqref{eq:iter_refine_m2} and~\eqref{eq:iter_refine} would
be arbitrary. In practice numerical errors cause loss of commutativity
which for some iterations results in instabilities leading to
unbounded growth of errors. This effect has been studied in several
papers for different variants of~\eqref{eq:iter_refine_m2} and for
other matrix iterations. For example, the iterations
\begin{equation} \label{eq:iter_refine_m2_alt1}
  Z_{i+1} = Z_i + \frac{1}{2}Z_i(I-Z_iSZ_i), 
\end{equation}
and
\begin{equation} \label{eq:iter_refine_m2_alt2}
  Z_{i+1} = \frac{3}{2}Z_i - \frac{1}{2}(Z_i^3 S),  
\end{equation}
considered in for example~\cite{Branislav_2007, Philippe_1987}
and~\cite{Bini_2005, Richters_2017}, respectively, are both equivalent
to~\eqref{eq:iter_refine_m2} in exact arithmetics.
However, both \eqref{eq:iter_refine_m2_alt1} and
\eqref{eq:iter_refine_m2_alt2} are unstable unless $S$ is extremely
well conditioned~\cite{Bini_2005, Philippe_1987}.
Iteration~\eqref{eq:iter_refine_m2} has been shown to be locally stable
around $S^{-1/2}$~\cite{Branislav_2007}. Here we will first
consider the stability of~\eqref{eq:iter_refine} around any inverse
factor $Z$ and for any $m$. Then we will consider the stability of the
local refinement given by~\eqref{eq:local_iter}.

Given a matrix function $f$, the Fr{\'e}chet derivative $L_f(X;E)$ at
$X$ in direction $E$ is a mapping, linear in $E$, satisfying
\begin{equation}
  L_f(X;E) = f(X+E)-f(X) + o(\|E\|)
\end{equation}
for all $E$~\cite{book-higham}.  We follow~\cite{cheng_2001,
  book-higham, HighamMackeyMackeyTisseur_2005} and define a matrix
iteration $X_{i+1} = f(X_i)$ to be stable in a neighborhood of a fixed
point $X=f(X)$ if $L_f(X;E)$ exists at $X$ and there is a constant $c$
such that $\max_{E\neq 0}\frac{\|L_f^p(X;E)\|}{\|E\|}\leq c$ for any
$p>0$.  Here, the $p$th power $L_f^p(X;E)$ is used to denote p-fold
composition of the Fr{\'e}chet derivative in the second argument,
e.g. $L_f^2(X;E) = L_f(X;L_f(X;E))$.

Let
\begin{equation}
  g_m(Z) = Z \sum_{k=0}^m b_k(I-Z^*SZ)^k
\end{equation}
be the mapping associated with~\eqref{eq:iter_refine}.  We note that
although $Z^*SZ=I$ implies $Z = g_m(Z)$, the converse is in general not
true. For example, if $Z^*SZ=I$, then both $Z=g_2(Z)$ and $\alpha Z
=g_2(\alpha Z)$ with $\alpha = \sqrt{7/3}$ are fixed points of $g_2$,
but $(\alpha Z)^*S(\alpha Z) \neq I$.
Let $Z^*SZ=I$. Then the Fr{\'e}chet derivative at $Z$,
\begin{equation}
  L_{g_m}(Z;E) = \frac{1}{2}(E-ZE^*SZ),
\end{equation}
is idempotent, i.e.~$L_{g_m}^2(Z;E) = L_{g_m}(Z;E)$,
and thus has bounded powers.

Let
\begin{equation}
  h_m(Z, \delta) =
  \begin{bmatrix}
    Z \sum_{k=0}^m b_k\delta^k
    \\ \delta - \left(Z \sum_{k=0}^m b_k\delta^k\right)^*S\left(Z \sum_{k=1}^m b_k\delta^k\right) -
    \left(Z \sum_{k=1}^m b_k\delta^k\right)^*SZ
  \end{bmatrix}
\end{equation}
be the mapping associated with~\eqref{eq:local_iter}.
Let $Z^*SZ=I$. Then the Fr{\'e}chet derivative at $(Z,0)$ in direction $(E,F)$
\begin{equation}
  L_{h_m}(Z,0;E,F) =
  \begin{bmatrix}
    E+\frac{1}{2}ZF \\
    \frac{1}{2}(F-F^*)
  \end{bmatrix}
\end{equation}
and
\begin{equation}
  L_{h_m}^p(Z,0;E,F) =
  \begin{bmatrix}
    E+\frac{1}{2}ZF + (p-1)\frac{1}{4}Z(F-F^*) \\
    \frac{1}{2}(F-F^*)
  \end{bmatrix}.
\end{equation}
We have that $L_{h_m}(Z,0;E,F)$ is idempotent if $F$ is
Hermitian. Thus,~\eqref{eq:local_iter} is stable if non-Hermitian
perturbations of $\delta$ are disallowed, which can simply be achieved
by storing only the upper triangle of $\delta$.

The stability properties of the matrix iterations above are
demonstrated in Figure~\ref{fig:stability}. Again, we
follow~\cite{book-higham} and apply the iterations to the Wilson
matrix
\begin{equation}
  S=\begin{bmatrix}10 & 7 & 8 & 7\\ 7 & 5 & 6 & 5\\ 8 & 6 & 10 & 9\\ 7 & 5 & 9 & 10\end{bmatrix}.
\end{equation}
For the local refinement without symmetric storage of $\delta$, an
unsymmetric perturbation in $\delta_0$ causes a small drift away from
the initial inverse factor. However, the factorization error stays
small.

\begin{figure}
  \begin{center}
    \begin{subfigure}{0.48\textwidth}
    \includegraphics[width=\textwidth]{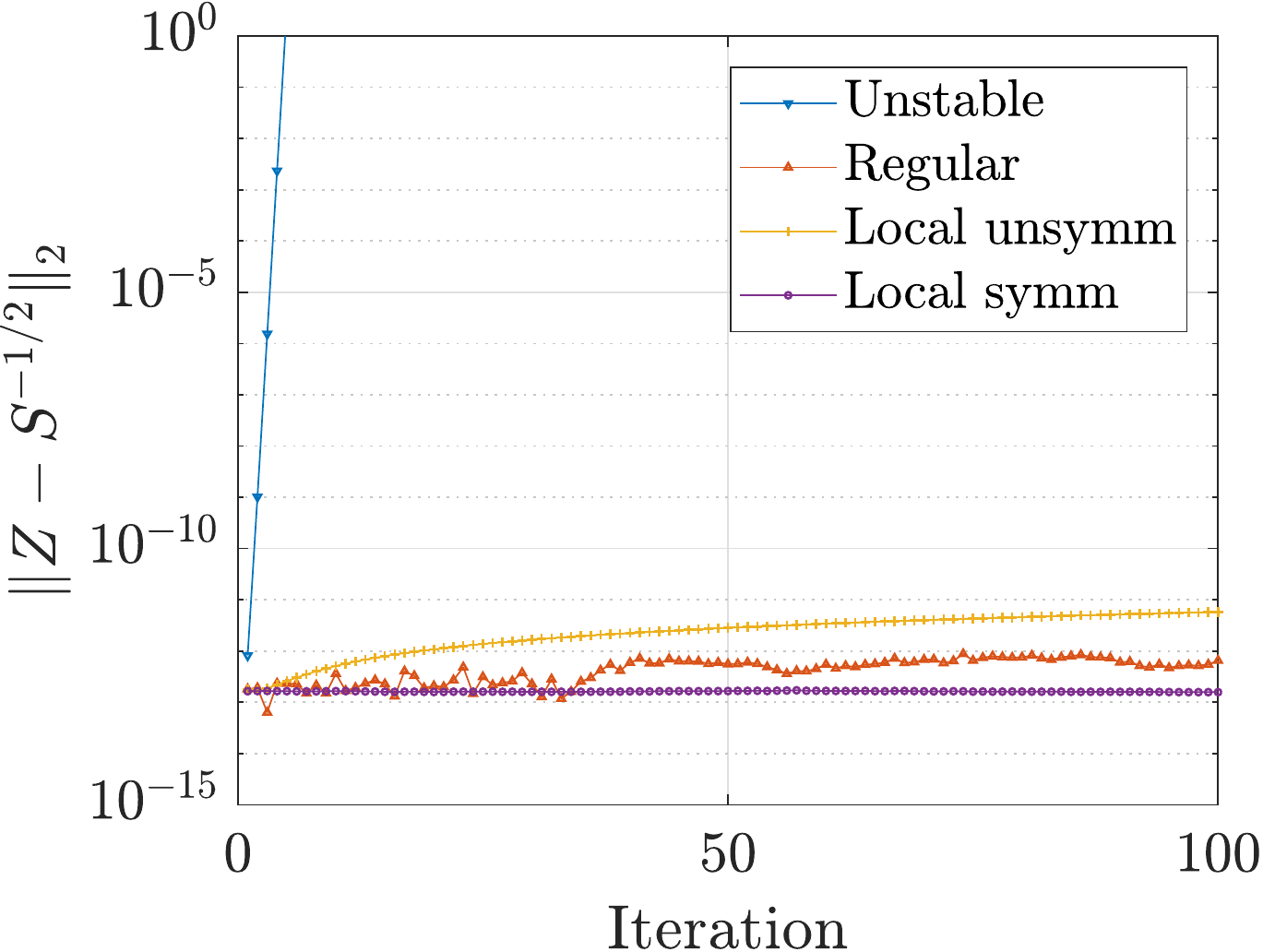}
    \end{subfigure}
    \begin{subfigure}{0.48\textwidth}
      \includegraphics[width=\textwidth]{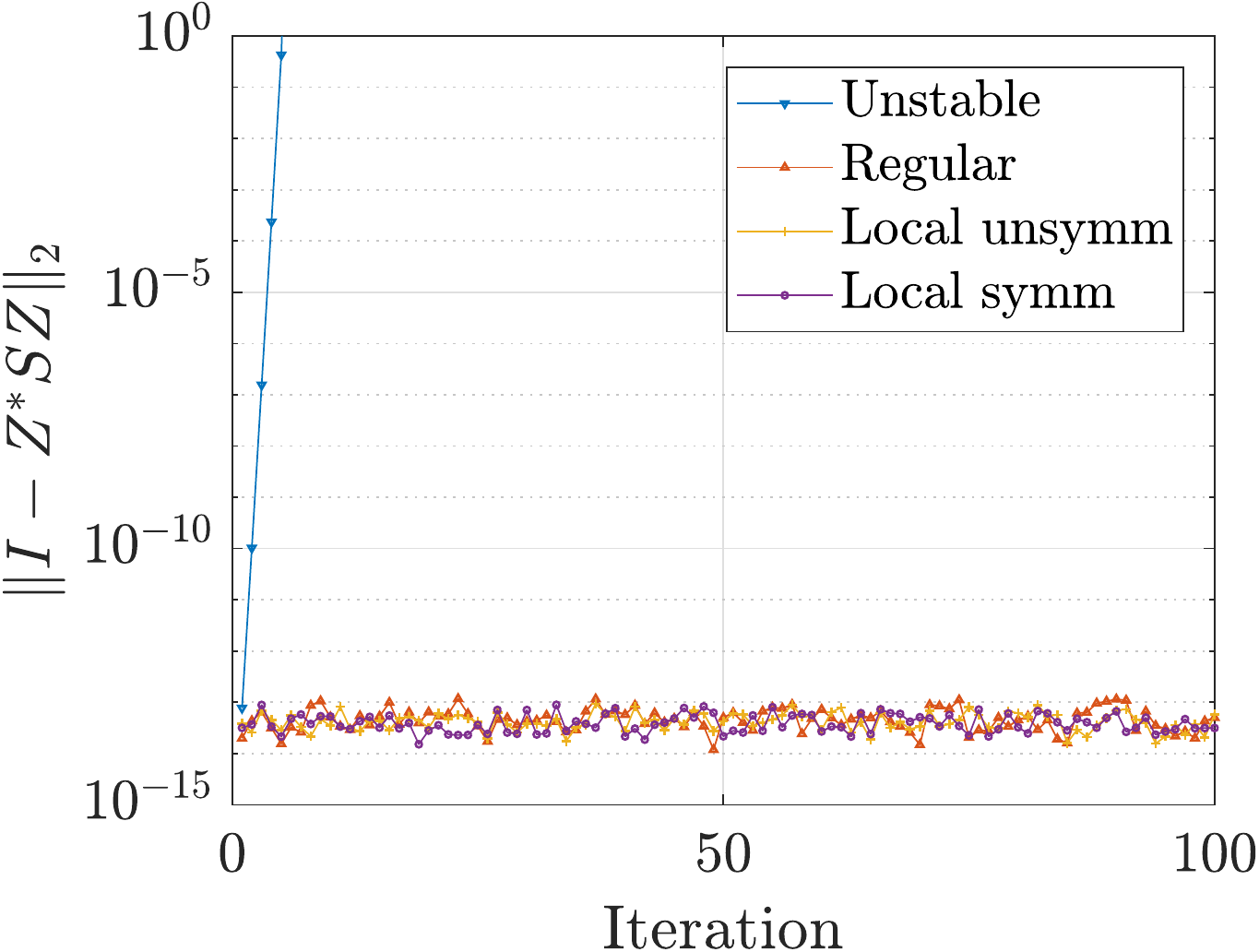}
    \end{subfigure}
  \end{center}
  \caption{Iterations \eqref{eq:iter_refine_m2_alt2}~(Unstable),
    \eqref{eq:iter_refine}~(Regular), and \eqref{eq:local_iter} with
    (Local symm) and without~(Local unsymm) symmetric storage of
    $\delta$ and $m=1$ applied to the Wilson matrix.  The iterations were
    started with $Z_0 = S^{-1/2}$ and $\delta_0 = I-Z_0^*SZ_0$
    computed in double precision arithmetics. Left panel: deviation
    from the inverse square root. Right panel: factorization
    error. \label{fig:stability} }
\end{figure}

\subsection{Stopping criterion}\label{sec:stopping}
Recently we proposed a new type of parameterless stopping criteria for
iterative methods~\cite{stop_crit_2016}. These stopping criteria were
originally used in recursive polynomial expansions to compute the
density matrix. The stopping criteria are based on the detection of a
discrepancy between theoretical and observed orders of convergence,
which takes place when numerical errors start to dominate the
calculation.  In other words, given some error measure or residual,
the theoretical worst case reduction of the error is derived. If the
error decreases any slower than this theoretical worst case reduction
the calculation has reached the stagnation phase, and it is time to
stop the iterations.

For the iterative refinement \eqref{eq:iter_refine} a worst case error
reduction is given by \eqref{eq:error_reduction}. We may therefore
formulate our stopping criterion for the iterative refinement as
\textit{stop as soon as $\|\delta_{i+1}\|_2 > \| \delta_i
  \|_2^{m+1}$}.

\subsubsection{Frobenius norm}
The spectral norm can be expensive to compute in practical
calculations.  In particular, near the iterative refinement
convergence the eigenvalues of $\delta_{i}$ may be clustered near 0,
making it difficult for an iterative eigensolver to compute the
spectral norm.  One may therefore want to use a cheaper alternative in
the stopping criterion. The computational cost of the Frobenius norm
is independent of the eigenvalue distribution and requires just one
pass over the matrix entries. The Frobenius norm is equal to
\begin{align}
  \|\delta_{i}\|_F = \sqrt{\sum_{j=1}^{n} \lambda_j^2 },
  \label{eq:frob_norm}
\end{align}
where $\{\lambda_j\}_{j=1}^n$ are the eigenvalues of $\delta_i$
ordered so that $|\lambda_j| \leq |\lambda_n| < 1$ for all $j$.

Since $\|\delta_i\|_2 \leq \|\delta_i\|_F$, we have that $\|\delta_i\|_F =
K_i \|\delta_i\|_2$ for some $K_i \geq 1$. Following the discussion
in~\cite{stop_crit_2016} we have that if
\begin{equation}
  K_{i+1} \leq K_{i}^{m+1},
  \label{eq:stopcrit_Kcond}
\end{equation}
then in exact arithmetics, using~\eqref{eq:error_reduction}, we have that
\begin{equation}
  \|\delta_{i+1}\|_F = K_{i+1}\|\delta_{i+1}\|_2 < K_i^{m+1} \|\delta_i\|_2^{m+1} = \|\delta_i\|_F^{m+1}, 
\end{equation}
and the stopping criterion for the iterative refinement can be
formulated as \textit{stop as soon as $\|\delta_{i+1}\|_F > \| \delta_i \|_F^{m+1}$}.

Now we will show that the condition~\eqref{eq:stopcrit_Kcond} is fulfilled.
Let $\beta_j = \sum_{k=m+1}^{2m+1} c_k \lambda_j^{k-1}$. Then, $\beta_j
\leq \beta_n, \ j = 1,2,\dots,n-1$ and we have
\begin{align}
  \|\delta_{i+1}\|_F & = \sqrt{\sum_{j=1}^{n} \left(\sum_{k=m+1}^{2m+1} c_k \lambda_j^k\right)^2} = \sqrt{\sum_{j=1}^{n} \beta_j^2\lambda_j^2} 
   \leq \beta_n  \sqrt{\sum_{j=1}^{n} \lambda_j^2}  = \beta_n \|\delta_i\|_F.
\end{align}
Moreover, the spectral norm can be written as
\begin{align}
  \|\delta_{i+1}\|_2 = \sum_{k=m+1}^{2m+1} c_k \lambda_n^k = \beta_n \lambda_n.
\end{align}
Then, we obtain
\begin{align}
  K_{i+1} & = \frac{\|\delta_{i+1}\|_F}{\|\delta_{i+1}\|_2}  \leq \frac{\beta_n \|\delta_i\|_F}{\beta_n\|\delta_{i}\|_2}  = K_i \leq K_i^{m+1}.
\end{align}

\section{Binary principal submatrix decomposition}\label{sec:binary_decomp}
The basic component of our recursive and localized inverse factorization algorithms is a binary principal submatrix decomposition of $S$.
Let $S$ be partitioned as
\begin{equation}
  S = \begin{bmatrix} A & B \\ B^* & C \end{bmatrix}
\end{equation}
and let $Z_A^* A Z_A = I$ and $Z_C^* C Z_C = I$ be inverse
factorizations of $A$ and $C$, respectively. 
Then, an inverse factor of $S$ can be computed using iterative
refinement with
\begin{equation} \label{eq:binary_partition_Z0}
  Z_0 = \begin{bmatrix} Z_A & 0 \\ 0 & Z_C \end{bmatrix}
\end{equation}
as starting guess, as described by Algorithm~\ref{alg:iter_refine}.
To be able to strictly bound the number of iterations in exact
arithmetics, we let Algorithm~\ref{alg:iter_refine} not make use of
the parameterless stopping criterion which relies on numerical errors
to decide when to stop. In practical calculations, the parameterless
stopping criterion is preferable and will be used in the formulation
of the full recursive and localized algorithms in
Section~\ref{sec:recursive}.

\begin{algorithm}
  \caption{Iterative refinement \label{alg:iter_refine}}
\begin{algorithmic}[1]
  \Procedure{iter-refine}{$S$, $Z_0$, $\varepsilon$}
  \State \textbf{input:} $S = \begin{bmatrix} A & B \\ B^* & C \end{bmatrix}$ and $Z_0 = \begin{bmatrix} Z_A & 0 \\ 0 & Z_C \end{bmatrix}$
  \State $\delta_{0} = I-Z_{0}^*SZ_{0}$
  \State $i = 0$
  \While{$\|\delta_{i}\|_2 > \varepsilon$}
  \State $Z_{i+1} = Z_{i} \sum_{k=0}^m b_k\delta_{i}^k$
  \State $\delta_{i+1} = I-Z_{i+1}^*SZ_{i+1}$
  \State $i = i+1$
  \EndWhile
  \State \Return $Z_{i}$
  \EndProcedure
\end{algorithmic}
\end{algorithm}

Algorithm~\ref{alg:iter_refine} as is already features localized
computations. The foregoing computation of $Z_A$ and $Z_C$ can be
performed as two separate computations without any interaction or
communication between them. The following iterative refinement employs
matrix-matrix multiplications for which implementations with good
performance usually exist, both for serial and parallel execution.
Furthermore, the matrices involved in the algorithm are typically
sparse with localized nonzero structure. If the sparse matrix-matrix
multiplications are performed using the locality-aware parallel
block-sparse matrix-matrix multiplication
of~\cite{LocalityAwareRubensson2016}, communication can be further
reduced.

We will now introduce 
modifications to Algorithm~\ref{alg:iter_refine} to further improve
its localization features and avoid both computation and communication
that is unnecessary.  Our localized refinement is given by two
modifications of Algorithm~\ref{alg:iter_refine}. Firstly, we make the
observation that
\begin{align}
  \delta_{0} & = I-Z_{0}^*SZ_{0} 
   = 
  \begin{bmatrix} 
    I - Z_A^*AZ_A & - Z_A^*BZ_C   \\ 
    -Z_C^*B^*Z_A   & I - Z_C^*CZ_C 
  \end{bmatrix} \\
  & = 
  -\begin{bmatrix} \label{eq:delta_0_new}
    0 & Z_A^*BZ_C   \\ 
    Z_C^*B^*Z_A   & 0
  \end{bmatrix}
\end{align}
and use~\eqref{eq:delta_0_new} for the computation of
$\delta_0$. Secondly, we use the local version of iterative refinement
given by~\eqref{eq:local_iter_practice}. Our localized iterative refinement is given by Algorithm~\ref{alg:local_refine}.

\begin{algorithm}
  \caption{Localized refinement \label{alg:local_refine}}
\begin{algorithmic}[1]
  \Procedure{local-refine}{$S$, $Z_0$, $\varepsilon$}
  \State \textbf{input:} $S = \begin{bmatrix} A & B \\ B^* & C \end{bmatrix}$ and $Z_0 = \begin{bmatrix} Z_A & 0 \\ 0 & Z_C \end{bmatrix}$
  \State $\delta_{0} = -\begin{bmatrix} 0 & Z_A^*BZ_C \\ Z_C^*B^*Z_A & 0 \end{bmatrix}$
  \State $i = 0$
  \While{$\|\delta_{i}\|_2 > \varepsilon$}
  \State $M_{i}  =  Z_{i} (\sum_{k=1}^m b_k\delta_{i}^k)$
  \State $Z_{i+1}  =  Z_{i} + M_{i}$
  \State $\delta_{i+1} = \delta_{i} - Z_{i+1}^*(SM_{i}) - (M_{i}^*S)Z_{i}$
  \State $i = i+1$
  \EndWhile
  \State \Return $Z_{i}$
  \EndProcedure
\end{algorithmic}
\end{algorithm}

Although Algorithms~\ref{alg:iter_refine} and~\ref{alg:local_refine}
are mathematically equivalent, their cost of execution and numerical
behavior is different. In the localized refinement the factorization
errors $I-Z_A^*AZ_A$ and $I-Z_C^*CZ_C$ are taken as zero and the
factorization error $\delta_{i+1}$ is in each iteration computed by
updating the error $\delta_{i}$ from the previous iteration. This
means that the algorithm is not capable of correcting for any initial
errors in $Z_A$ and $Z_C$ nor for any errors introduced while updating
$\delta_{i+1}$. This stands in contrast to
Algorithm~\ref{alg:iter_refine} where the factorization error is
recomputed in each iteration. Both algorithms are stable though, as
shown in Section~\ref{sec:stability}. Another drawback with
Algorithm~\ref{alg:local_refine} is that it requires more
matrix-matrix multiplications per iteration. With $m=1$,
Algorithm~\ref{alg:iter_refine} and Algorithm~\ref{alg:local_refine}
make use of~3 and 4~multiplications per iteration, respectively,
assuming that the equality $(M_i^*S) = (SM_i)^*$ is used to avoid
1~multiplication in Algorithm~\ref{alg:local_refine}. Nevertheless, a
great advantage of Algorithm~\ref{alg:local_refine} is its
localization properties. Although both algorithms feature localized
computations in some sense, we will see in
Section~\ref{sec:localization} that the localized refinement is
superior for large systems with localization in $S$.

It was shown in~\cite{rubenssonBockHolmstromNiklasson} that with the
starting guess given by~\eqref{eq:binary_partition_Z0}, we always get
an initial factorization error $\|\delta_0\|_2 < 1$ and convergence of
the iterative refinement, regardless of what inverse factors $Z_A$ and
$Z_C$ that are used. The following theorem is a strengthening of this
result giving quantitative insight into how the convergence depends on
the eigenvalues or condition number of $S$.

\begin{theorem}\label{thm:recinvfact}
  Let $S$ be a Hermitian positive definite $n \times n$ matrix partitioned as
\begin{align}
  S = \begin{bmatrix} A & B \\ B^* & C \end{bmatrix}
\end{align}
where $A$ is $n_A \times n_A$ with $1 \leq n_A < n$.  Let
$Z_A^* A Z_A = I$ and $Z_C^* C Z_C = I$ be inverse factorizations of
$A$ and $C$ and let
\begin{equation}
  Z_0 = \begin{bmatrix} Z_A & 0 \\ 0 & Z_C \end{bmatrix}.
\end{equation}
Then,
\begin{align}
&\lambda_{\mathrm{min}}(S) \leq \lambda_i(A) \leq \lambda_{\mathrm{max}}(S), &i &= 1, \dots, n_A, \label{eq:thm_rif_1}\\
&\lambda_{\mathrm{min}}(S) \leq \lambda_i(C) \leq \lambda_{\mathrm{max}}(S), &i &= 1, \dots, n-n_A, \label{eq:thm_rif_2}\\
&\frac{\lambda_{\mathrm{min}}(S)}{\lambda_{\mathrm{max}}(S)}-1 \leq \lambda_i(I-Z_0^*SZ_0) \leq 1-\frac{\lambda_{\mathrm{min}}(S)}{\lambda_{\mathrm{max}}(S)}, &i &= 1, \dots, n. \label{eq:thm_rif_3}
\end{align}
\end{theorem}

Theorem~\ref{thm:recinvfact} implies $\|\delta_0\|_2 < 1$ and
convergence of the iterative refinement in
Algorithms~\ref{alg:iter_refine} and~\ref{alg:local_refine}. It follows
from~\eqref{eq:no-of-iters} and~\eqref{eq:thm_rif_3} that for a given
level of accuracy $\|\delta_{i}\|_2 < \varepsilon$ the number of
iterations is bounded by
\begin{equation}\label{eq:iter_refine_no_of_iters}
  k = \left\lceil\frac{\log{\left(\frac{\log \varepsilon}{\log \left(1-{\lambda_{\mathrm{min}}(S)}/{\lambda_{\mathrm{max}}(S)}\right)}\right)}}{\log{(m+1)}}\right\rceil.
\end{equation}

\begin{lemma}\label{thm:ostrowski}
  Let $A$ and $B$ be positive definite Hermitian $n \times n$
  matrices. Then $AB$ has real eigenvalues and
  \begin{equation}
    \lambda_{\mathrm{min}}(A) \lambda_i(B) \leq \lambda_i(AB) \leq \lambda_{\mathrm{max}}(A) \lambda_i(B), \quad i = 1,\dots,n.
  \end{equation}  
\end{lemma}
\begin{proof}
  Lemma~\ref{thm:ostrowski} is a specialization of a theorem due to
  Ostrowski~\cite[Satz~1]{VonOstrowski_1959}.
\end{proof}
\begin{proof}[Proof of Theorem~\ref{thm:recinvfact}]
  The inequalities~\eqref{eq:thm_rif_1} and~\eqref{eq:thm_rif_2}
  follow directly from Cauchy's interlacing
  theorem, see e.g.~\cite{book-matrix-analysis}. 
  Recall that $SZ_0Z_0^*$ has the same eigenvalues as $Z_0^*SZ_0$ and invoke
  Lemma~\ref{thm:ostrowski} with $A = S$ and $B = Z_0Z_0^*$. This gives
  \begin{align}
    \lambda(Z_0^*SZ_0) & \geq \lambda_{\mathrm{min}}(S)\lambda_{\mathrm{min}}(Z_0Z_0^*) \\
    & = \lambda_{\mathrm{min}}(S)\min\left(\lambda_{\mathrm{min}}(A^{-1}), \lambda_{\mathrm{min}}(C^{-1})\right) \\
    & = \lambda_{\mathrm{min}}(S)\min\left(\frac{1}{\lambda_{\mathrm{max}}(A)}, \frac{1}{\lambda_{\mathrm{max}}(C)}\right) \\
    & \geq \lambda_{\mathrm{min}}(S)\frac{1}{\lambda_{\mathrm{max}}(S)}
  \end{align}
  where we used that $Z_0Z_0^* = \begin{bmatrix} A^{-1} & 0 \\ 0 &
    C^{-1} \end{bmatrix}$ and again Cauchy's interlacing theorem. This
  gives us the upper bound in~\eqref{eq:thm_rif_3}. The lower bound follows from the fact that $I-Z_0^*SZ_0$ is a so-called Jordan-Wielandt matrix with positive-negative eigenvalue pairs, see e.g.~\cite{book-stewart-sun}.
\end{proof}

Before we present our complete localized inverse factorization
algorithm we will theoretically investigate localization properties of
Algorithms~\ref{alg:iter_refine} and~\ref{alg:local_refine}.  In
particular, we will under certain assumptions show that the localized
refinement involves only operations on matrices with a number of
significant entries proportional to the size of the cut that defines
the principal submatrix decomposition.

\section{Localization}\label{sec:localization}
Let $d(i,j)$ be a pseudometric on the index set $I_S = \{1\dots n\}$ of $S$,
i.e.~a distance function such that $d(i,j) \geq 0$, $d(i,j) = d(j,i)$, $d(i,i) = 0$, and $d(i,j)\leq d(i,k)+d(k,j)$ hold for
all $i$, $j$, and $k$.
Let $N_d(i,R) = \{j:d(i,j) < R\}$ be the set of vertices within
distance $R$ from $i$ and let $|N_d(i,R)|$ denote its cardinality.
In case $S$ is a basis set overlap matrix for a basis set with atom
centered basis functions, the vertices (indices) correspond to basis
function centers and $d(i,j)$ may be naturally taken as the Euclidean
distance between the centers corresponding to $i$ and $j$.

\subsection{Exponential decay with distance between vertices}
We will say that an $n \times n$ matrix $S$, with associated distance
function $d(i,j)$, has the property of \emph{exponential decay with
  respect to distance between vertices with constants $c$ and
  $\alpha$} if
\begin{equation}\label{eq:exp_decay_single_matrix}
  |S_{ij}| \leq ce^{-\alpha d(i,j)} 
\end{equation}
for all $i,j=1,\dots,n$ with $c>0$ and $\alpha>0$.  We shall in
particular be concerned with sequences of matrices $\{S_n\}$ with
associated distance functions $\{d_n(i,j)\}$ that satisfy exponential
decay with respect to distance between vertices
\eqref{eq:exp_decay_single_matrix} with constants $c$ and $\alpha$
independent of $n$.
\begin{theorem}\label{thm:on_entries_with_expdecay}
Let $\{S_n\}$ be a sequence of $n \times n$ matrices with associated
distance functions $\{d_n(i,j)\}$ and assume that each $S_n$ satisfies the 
exponential decay property
\begin{equation}\label{eq:thm4_exp_decay}
  |[S_n]_{ij}| \leq ce^{-\alpha d_n(i,j)} 
\end{equation}
for all $i,j=1,\dots,n$ with constants $c>0$ and $\alpha>0$
independent of $n$.
Assume that there
are constants $\gamma>0$ and $\beta > 0$ independent of $n$ such that
\begin{equation}\label{eq:bound_n_neighbor}
  |N_{d_n}(i,R)| \leq \gamma R^\beta
\end{equation}
holds for all $i$, for any $R \geq 0$. Then, for any given $\varepsilon > 0$, 
each $S_n$ contains at most $O(n)$ entries greater
than $\varepsilon$ in magnitude.
Also, the number of entries greater than $\varepsilon$ in magnitude in each
row and column of each $S_n$ is bounded by a constant independent of $n$.
\end{theorem}
\begin{proof}
  For any matrix entry $[S_n]_{ij}$ with magnitude greater than $\varepsilon$,
  \eqref{eq:thm4_exp_decay} implies
  \begin{equation}
    d_n(i,j) < \frac{1}{\alpha}\ln \left(\frac{c}{\varepsilon}\right)
  \end{equation}
  which is a constant independent of $n$. For each vertex $i$ the
  number of vertices within a constant distance is, due
  to~\eqref{eq:bound_n_neighbor}, bounded by a constant independent of
  $n$.  For any row or column in the matrix the number of entries
  larger than $\varepsilon$ is therefore bounded by a constant and the
  total number of entries satisfying $|[S_n]_{ij}| > \varepsilon$
  cannot exceed $O(n)$.
\end{proof}

\begin{theorem}\label{thm:exp_decay_product} 
  Let $\{A_n\}$ and $\{B_n\}$ be sequences of $n \times n$ matrices
  with a common associated distance function $d_n(i,j)$ for each $n$.
  Assume that
  \begin{align}
    |[A_n]_{ij}| &\leq c_1e^{-\alpha d_n(i,j)}, \\
    |[B_n]_{ij}| &\leq c_2e^{-\alpha d_n(i,j)} 
  \end{align}
  for all $i,j$ where $c_1$, $c_2$, and $\alpha$ are positive and
  independent of $n$.  Assume that there are
  constants $\gamma>0$ and $\beta > 0$ independent of $n$ such that
  \begin{equation} \label{eq:bound_n_neighbor_B}
    |N_{d_n}(i,R)| \leq \gamma R^\beta
  \end{equation}
  holds for all $i$, for any $R \geq 0$. Then, the entries of $C_n = A_nB_n$ satisfy
  \begin{equation}
    |[C_n]_{ij}| \leq ce^{-\alpha' d_n(i,j)} \textrm{ for all } i,j
  \end{equation}
  for any $\alpha'$ such that $0<\alpha'<\alpha$ with $c$ independent of $n$.
\end{theorem}
\begin{proof}
  Let $\omega = \alpha-\alpha'$ and note that $\alpha' < \alpha$. Then,
  \begin{align}
    |[A_n]_{ij}| &\leq c_1e^{-\alpha' d_n(i,j)}, \\
    |[B_n]_{ij}| &\leq c_2e^{-(\omega+\alpha') d_n(i,j)}
  \end{align}
  which gives 
  \begin{align}
    |[C_n]_{ij}| &\leq \sum_{k=1}^n |[A_n]_{ik}||[B_n]_{kj}| \\
    &\leq 
    c_1c_2\sum_{k=1}^n e^{-\alpha' (d_n(i,k)+d_n(k,j))} e^{-\omega d_n(k,j)}\\
    &\leq c_1c_2\sum_{k=1}^n e^{-\omega d_n(k,j)}e^{-\alpha' d_n(i,j)}.
  \end{align}
  So far we have essentially followed the proof of Theorem~9.2
  in~\cite{benzi-decay}.  It remains to show that the sum
  $\sum_{k=1}^n e^{-\omega d_n(k,j)}$ is bounded by a constant
  independent of $n$.  Grouping the summands with respect to distance
  from vertex $j$
  gives
  \begin{align}
    \sum_{k=1}^n e^{-\omega d_n(k,j)} &\leq \sum_{r=1}^\infty (|N_{d_n}(j,r)|-|N_{d_n}(j,r-1)|)e^{-\omega (r-1)}\\
    &\leq \sum_{r=1}^\infty |N_{d_n}(j,r)|e^{-\omega (r-1)}\\
    &\leq \sum_{r=1}^\infty \gamma r^\beta e^{-\omega (r-1)}.\label{eq:bound_for_mult_decay}
  \end{align}
  Note that $|N_{d_n}(j,r)|-|N_{d_n}(j,r-1)|$ is the number of
  vertices located at a distance $r-1 \leq d_n(k,j) < r$ from vertex
  $j$.
  The sum~\eqref{eq:bound_for_mult_decay} is independent of $n$ and
  can be shown to converge using the ratio test~\cite[p.~66]{book-rudin}.
 \end{proof}

As already noted, the results of this subsection are closely related
to results previously presented for example in~\cite{benzi-decay}. See
in particular Proposition~6.4 and Theorem~9.2 of~\cite{benzi-decay}. A
key difference, however, is that in~\cite{benzi-decay} the distance
function or metric on the index set $I_S$ is assumed to be the
geodesic distance function of a graph defined over $I_S$.
In this sense the present formulation, where any pseudometric may be
used, is more general.
However, in~\cite{benzi-decay} a less restrictive condition for the
number of neighbors of any node is used, in comparison with
\eqref{eq:bound_n_neighbor} and
\eqref{eq:bound_n_neighbor_B}. In~\cite{benzi-decay} it is assumed
that the maximum degree of the graph, i.e.\ the largest number of
immediate neighbors of any vertex, is uniformly bounded with respect
to $n$.
To see that this is a less restrictive condition, consider for example
the graph given by a binary tree with maximum degree 3. For any
constants $\gamma$ and $\beta$ there exist $n$, $R$, and $i$ such that
$|N_{d_n}(i,R)| > \gamma R^\beta$ since $\max_i|N_{d_n}(i,R)|$ grows
exponentially with $R$ for large enough $n$, so that
\eqref{eq:bound_n_neighbor} and \eqref{eq:bound_n_neighbor_B} are
violated.
However, in calculations with the pseudometric taken as Euclidean
distance between atom centered basis functions this is not an issue
since the number of basis function centers within a certain distance
$R$ will never exceed $O(R^3)$.

\subsection{Exponential decay away from cut}
We will here consider matrices with the property of exponential decay
away from a set of indices $I\subseteq \{1,\dots,n\}$. The decay may be with respect to row index
\begin{equation}
  |S_{ij}| \leq ce^{-\alpha \min_{k\in I}d(i,k)} \textrm{ for all } i,j=1,\dots,n,  
\end{equation}
column index 
\begin{equation}
  |S_{ij}| \leq ce^{-\alpha \min_{k\in I}d(j,k)} \textrm{ for all } i,j=1,\dots,n,  
\end{equation}
or both. We are in particular interested in binary partitions of the
index set, i.e.~$I_A\subset \{1,\dots,n\}$ and $I_C =
\{1,\dots,n\}\setminus I_A$, corresponding to a binary principal
submatrix partition $S = \begin{bmatrix} A & B \\ D &
  C \end{bmatrix}$. For such partitions we define the distance to the
cut
\begin{equation}\label{eq:cut_dist_single_matrix}
  d^\mathrm{cut}(i) = \min_{k\in I_A} d(i,k) + \min_{k\in I_C} d(i,k).
\end{equation}
Note that for $i \in I_A$ the first term on the right hand side will
be zero and the distance to the cut is thus defined as the distance to
the closest vertex in $I_C$, and vice versa for $i \in I_C$. This is
illustrated in Figure~\ref{fig:cut_distance}.  We will say that $S$
has the property of \emph{exponential decay away from the cut with
  constants $c$ and $\alpha$} if
\begin{equation} \label{eq:exp_decay_cut}
  |S_{ij}| \leq ce^{-\alpha \max(d^\mathrm{cut}(i),d^\mathrm{cut}(j))}
\end{equation}
for all $i,j=1,\dots,n$ with $c>0$ and $\alpha>0$.
We note that this is equivalent to the elements of $S$ satisfying
the four conditions
\begin{align}
  |S_{ij}| & \leq ce^{-\alpha \min_{k\in I_A}d(i,k)}, \quad 
  & |S_{ij}| & \leq ce^{-\alpha \min_{k\in I_C}d(i,k)}, \nonumber \\ 
  |S_{ij}| & \leq ce^{-\alpha \min_{k\in I_A}d(j,k)}, \textrm{ and }\quad 
  & |S_{ij}| & \leq ce^{-\alpha \min_{k\in I_C}d(j,k)} \label{eq:cut_decay_equiv}
\end{align}
for all $i,j=1,\dots,n$.
We shall in particular be concerned with sequences of matrices
$\{S_n\}$ where each matrix $S_n$ is associated with a distance
function $d_n(i,j)$, a binary partition of its index set,
i.e.~$I_{A_n}\subset \{1,\dots,n\}$ and $I_{C_n} =
\{1,\dots,n\}\setminus I_{A_n}$, and 
the distance to the cut, defined as in \eqref{eq:cut_dist_single_matrix},
\begin{equation} \label{eq:cut_dist_sequence}
  d_n^\mathrm{cut}(i) = \min_{k\in I_{A_n}} d_n(i,k) + \min_{k\in I_{C_n}} d_n(i,k).
\end{equation}

\begin{figure}
  \begin{center}
    \includegraphics[scale=1.04]{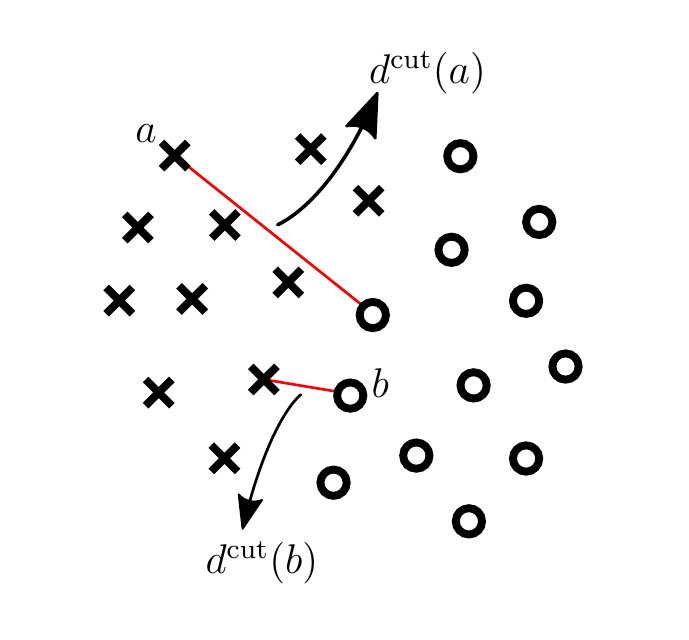}
  \end{center}
  \caption{Illustration of distance to cut. The vertex (index) set
    $\{1,\dots,n\}$ is partitioned in two subsets $I_A$ (crosses) and
    $I_C$ (circles).  For a vertex $a$ in $I_A$, the distance to the
    cut $d^\mathrm{cut}(a)$ is defined as the distance to the closest
    vertex in $I_C$. For a vertex $b$ in $I_C$, the distance to the
    cut $d^\mathrm{cut}(b)$ is defined as the distance to the closest
    vertex in $I_A$. \label{fig:cut_distance} }
\end{figure}

\begin{theorem}\label{thm:k_entries_with_cutdecay}
  Let $\{S_n\}$ be a sequence of $n\times n$ matrices satisfying the
  assumptions of Theorem~\ref{thm:on_entries_with_expdecay}. Associate
  with each $S_n$ a binary partition of its index set and let
  $d_n^\mathrm{cut}(i)$ be defined as in \eqref{eq:cut_dist_sequence}.
  Assume furthermore that for each distance $R$, there are constants
  $\gamma>0$ and $\beta>0$ independent of $n$ and a function $p(n)$
  such that
  \begin{equation}\label{eq:entries_within_R_from_cut}
    |\{i:d_n^\mathrm{cut}(i) < R \}| < \gamma R^\beta p(n),
  \end{equation}
  i.e.~the number of vertices within distance $R$ from the cut is
  bounded by $\gamma R^\beta p(n)$.
  Assume also that at least one of
  \begin{equation} \label{thm:decay_from_cut}
    |[S_n]_{ij}| \leq ce^{-\alpha d_n^\mathrm{cut}(i)} 
    \quad \textrm{and} \quad
    |[S_n]_{ij}| \leq ce^{-\alpha d_n^\mathrm{cut}(j)} 
  \end{equation}
  hold for all $i,j=1,\dots,n$.
  Then, for any $\varepsilon > 0$ each $S_n$ contains at most
  $O(p(n))$ entries greater than $\varepsilon$ in magnitude.
\end{theorem}
\begin{proof}
  For any matrix entry $[S_n]_{ij}$ with magnitude greater than $\varepsilon$,
  \eqref{thm:decay_from_cut} implies 
  \begin{equation}
    d_n^\mathrm{cut}(i) < \frac{1}{\alpha}\ln \left(\frac{c}{\varepsilon}\right)
    \quad \textrm{or} \quad
    d_n^\mathrm{cut}(j) < \frac{1}{\alpha}\ln \left(\frac{c}{\varepsilon}\right).
  \end{equation}
  By~\eqref{eq:entries_within_R_from_cut} the number of vertices 
  within distance $\frac{1}{\alpha}\ln
  \left(\frac{c}{\varepsilon}\right)$ from the cut
  is bounded by $\gamma\left(\frac{1}{\alpha}\ln
  \left(\frac{c}{\varepsilon}\right)\right)^\beta p(n)$, where
  $\gamma\left(\frac{1}{\alpha}\ln
  \left(\frac{c}{\varepsilon}\right)\right)^\beta$ is a constant independent
  of $n$.
  Thus, only $O(p(n))$ rows or columns may have entries
  with magnitude greater than $\varepsilon$ and
  by Theorem~\ref{thm:on_entries_with_expdecay}, the number of 
  entries in each row and column with magnitude greater than
  $\varepsilon$ is bounded by a constant.
\end{proof}
We will refer to the number of vertices within a given distance $R$
from the cut, i.e. $|\{i:d_n^\mathrm{cut}(i) < R \}|$, as the
\emph{cut size}.  The function $p(n)$ in
\eqref{eq:entries_within_R_from_cut} describes how the cut size
increases with $n$.
Theorem~\ref{thm:k_entries_with_cutdecay} tells us that for a
sequence of matrices with exponential decay with distance between
vertices and exponential decay away from the cut, the number of
significant entries does not grow faster than the cut size. Note that,
in general, exponential decay away from the cut alone is not
sufficient to reach this result.
For example, assume that $p(n) = \sqrt{n}$, so that the cut size grows
as $O(\sqrt{n})$, and that
both conditions in \eqref{thm:decay_from_cut} are satisfied.  Then,
$O(\sqrt{n})$ rows and $O(\sqrt{n})$ columns may have entries with
magnitude greater than $\varepsilon$, giving a total of $O(n)$ matrix
entries that may have magnitude greater than $\varepsilon$.

\begin{theorem}\label{thm:exp_decay_from_index_set}
  Let $\{S_n\}$ and $\{T_n\}$ be sequences of $n \times n$ matrices
  with a common associated distance function $d_n(i,j)$ for each $n$.
  Assume that for any $R \geq 0$, there are constants $\gamma>0$ and $\beta
  > 0$ independent of $n$ such that
  \begin{equation} \label{eq:bound_n_neighbor_C}
    |N_{d_n}(i,R)| \leq \gamma R^\beta
  \end{equation}
  holds for all $i$. For each $n$, let $I_n$ be a subset of the index set
  $\{1,\dots,n\}$ and let $c_1$, $c_2$, and $\alpha$ be positive and
  independent of $n$.
  \begin{enumerate}
  \item[(i)]   Assume that
    \begin{align}
      |[S_n]_{ij}| &\leq c_1e^{-\alpha \min_{p\in I_n}d_n(i,p)}, \\
      |[T_n]_{ij}| &\leq c_2e^{-\alpha d_n(i,j)} 
    \end{align}
    for all $i,j$. 
    Then, the entries of $U_n = S_nT_n$ satisfy
    \begin{equation}\label{eq:mmmul_cut_decay_i}
      |[U_n]_{ij}| \leq ce^{-\alpha' \min_{p\in I_n}d_n(i,p)} \textrm{ for all } i,j
    \end{equation}
    for any $\alpha'$ such that $0<\alpha'<\alpha$ with $c$
    independent of $n$. 
    This bound holds also with $U_n = T_nS_n$.
  \item[(ii)]   Assume that
    \begin{align}
      |[S_n]_{ij}| &\leq c_1e^{-\alpha \min_{p\in I_n}d_n(j,p)}, \\
      |[T_n]_{ij}| &\leq c_2e^{-\alpha d_n(i,j)} 
    \end{align}
    for all $i,j$. Then, the entries of $U_n = S_nT_n$ satisfy
    \begin{equation}\label{eq:mmmul_cut_decay_j}
      |[U_n]_{ij}| \leq ce^{-\alpha' \min_{p\in I_n}d_n(j,p)} \textrm{ for all } i,j
    \end{equation}
    for any $\alpha'$ such that $0<\alpha'<\alpha$ with $c$ independent
    of $n$. 
    This bound holds also with $U_n = T_nS_n$.
  \end{enumerate}
\end{theorem}
\begin{proof}
  Case~(i): Similarly to the proof of
  Theorem~\ref{thm:exp_decay_product} we have, again with
  $\omega=\alpha-\alpha'$, that
  \begin{align}
    |[U_n]_{ij}| &\leq \sum_{k=1}^n |[S_n]_{ik}||[T_n]_{kj}| \\
    &\leq 
    c_1c_2\sum_{k=1}^n e^{-\alpha' (\min_{p\in I_n}d_n(i,p)+d_n(k,j))} e^{-\omega d_n(k,j)}\\
    &\leq c_1c_2\sum_{k=1}^n e^{-\omega d_n(k,j)}e^{-\alpha' \min_{p\in I_n}d_n(i,p)}
  \end{align}
  where $\sum_{k=1}^n e^{-\omega d_n(k,j)}$ is a constant independent
  of $n$, due to \eqref{eq:bound_n_neighbor_C}, as shown in the proof of
  Theorem~\ref{thm:exp_decay_product}.  The case with $U_n = T_nS_n$
  and case~(ii) can be shown in essentially the same way using that
  $\min_{p\in I_n}d_n(k,p)+d_n(k,j) \geq \min_{p\in I_n}d_n(j,p)$.
\end{proof}

We are particularly interested in the case of a binary partition of
the index set and exponential decay away from the cut. The following
result shows that multiplication of a matrix with exponential decay
away from the cut with a matrix with exponential decay with distance
between vertices gives a matrix with exponential decay away from the
cut.

\begin{corollary}\label{cor:cut_decay_product}
  Let $\{S_n\}$ and $\{T_n\}$ be sequences of $n \times n$ matrices
  with a common associated distance function $d_n(i,j)$ for each $n$.
  Assume that for any $R \geq 0$, there are constants $\gamma>0$ and $\beta
  > 0$ independent of $n$ such that
  \begin{equation}
    |N_{d_n}(i,R)| \leq \gamma R^\beta
  \end{equation}
  holds for all $i$. For each $n$, let $I_{A_n}$, $I_{C_n}$ be a binary partition of the
  index set $\{1,\dots,n\}$ and assume that
  \begin{align}
    |[S_n]_{ij}| &\leq c_1e^{-\alpha \max (d^\mathrm{cut}_n(i),d^\mathrm{cut}_n(j))}, \\
    |[T_n]_{ij}| &\leq c_2e^{-\alpha d_n(i,j)} 
  \end{align}
  for all $i,j$ where $d_n^\mathrm{cut}(i) = \min_{k\in I_{A_n}} d_n(i,k)
  + \min_{k\in I_{C_n}} d_n(i,k)$ and where $c_1$, $c_2$, and $\alpha$ are
  positive and independent of $n$. Then, the entries of $U_n = S_nT_n$
  satisfy
  \begin{equation}
    |[U_n]_{ij}| \leq ce^{-\alpha' \max (d^\mathrm{cut}_n(i),d^\mathrm{cut}_n(j))} \textrm{ for all } i,j
  \end{equation}
  for any $\alpha'$ such that $0<\alpha'<\alpha$ with $c$ independent of $n$.
  This bound holds also with $U_n = T_nS_n$.
\end{corollary}
\begin{proof}
  Using that exponential decay away from the cut is equivalent to the
  four conditions~\eqref{eq:cut_decay_equiv} the result follows
  directly from Theorem~\ref{thm:exp_decay_from_index_set}.
\end{proof}

\subsection{Localization in iterative refinement}
In this section we will derive localization results for the matrices
occurring in regular and localized iterative refinement. We will under
certain assumptions show that while both
Algorithms~\ref{alg:iter_refine} and~\ref{alg:local_refine} involve
only matrices that satisfy exponential decay with respect to distance
between vertices, all matrices constructed in Algorithm~\ref{alg:local_refine}
also satisfy exponential decay with respect to distance from the
cut. 

\begin{theorem}\label{thm:iter_refine_preserves_decay}
Let $\{S_n\}$ and $\{Q_n\}$ be sequences of $n \times n$
matrices with associated distance functions
$\{d_n(i,j)\}$. Let each $S_n$ be Hermitian and partitioned as
\begin{equation}
  S_n = \begin{bmatrix} A_n & B_n \\ B_n^* & C_n \end{bmatrix}
\end{equation}
and let
\begin{equation}
  Q_n = \begin{bmatrix} Z_{A_n} & 0 \\ 0 & Z_{C_n} \end{bmatrix}
\end{equation}
where $Z_{A_n}$ and $Z_{C_n}$ satisfy $Z_{A_n}^* A_n Z_{A_n} = I$ and $Z_{C_n}^* C_n Z_{C_n} = I$, respectively.
Let $\{Z_n\}$ be the sequence of matrices produced by Algorithm~\ref{alg:iter_refine} or Algorithm~\ref{alg:local_refine}
with $S_n$, $Q_n$, and a constant $\varepsilon$ as input.
For each iteration $i=0,1,\dots$, until the stopping criterion $\|\delta_i\|_2 \leq \varepsilon$ is triggered, let 
\begin{align}
  \label{eq:alg_matrices_first}
  \{(Z_{i})_n\}, & 
\\
  \{(\delta_{i}^k)_n\}, & \quad k=1,2,\dots,m,
  \label{eq:alg_matrices_second}
\\
  \{(M_{i})_n\}, 
\\
  \{(SM_{i})_n\}, 
\\
  \{(Z_{i+1}^*(SM_{i}))_n\}, 
\\
  \{((M_{i}^*S)Z_{i})_n\}, 
\label{eq:alg_matrices_last}
\end{align}
be the sequences of matrices corresponding to each of the intermediate
matrices occurring in either one or both of
Algorithm~\ref{alg:iter_refine} and Algorithm~\ref{alg:local_refine}.
Assume that for any $R \geq 0$, there are constants $\gamma>0$ and
$\beta > 0$ independent of $n$ such that
\begin{equation}
  |N_{d_n}(i,R)| \leq \gamma R^\beta
\end{equation}
holds for all $i$.
Assume that $S_n$ and $Q_n$ satisfy the exponential
decay properties
\begin{align}
  |[S_n]_{ij}| & \leq ce^{-\alpha d_n(i,j)} \\
  |[Q_n]_{ij}| & \leq ce^{-\alpha d_n(i,j)} 
\end{align}
for all $i,j=1,\dots,n$ with constants $c>0$ and $\alpha>0$
independent of $n$. Assume also that the condition number of $S_n$,
$\kappa_n = |\lambda_{\mathrm{max}}(S_n) /
\lambda_{\mathrm{min}}(S_n)|$, is uniformly bounded with respect to
$n$.

Then, each of the matrices in \eqref{eq:alg_matrices_first} through
\eqref{eq:alg_matrices_last} satisfies an exponential decay property on the form
\begin{equation} \label{eq:alg_decay_1}
  |[X_n]_{ij}| \leq c'e^{-\alpha' d_n(i,j)}  \textrm{ for all } i,j
\end{equation}
for any $\alpha'$ such that $0<\alpha'<\alpha$ with $c'$ independent
of $n$, where $X_n$ is any of the matrices in
\eqref{eq:alg_matrices_first} through \eqref{eq:alg_matrices_last}.
Besides satisfying \eqref{eq:alg_decay_1}, the matrices in
\eqref{eq:alg_matrices_second} through \eqref{eq:alg_matrices_last}
also satisfy exponential decay away from the cut on the form 
\begin{equation}
  |[X_n]_{ij}| \leq c'e^{-\alpha' \max (d^\mathrm{cut}_n(i),d^\mathrm{cut}_n(j))} \textrm{ for all } i,j
\end{equation}
for any $\alpha'$ such that $0<\alpha'<\alpha$ with $c'$ independent
of $n$, where $X_n$ is any of the matrices in
\eqref{eq:alg_matrices_second} through \eqref{eq:alg_matrices_last}.

\end{theorem}

\begin{lemma}\label{lem:exp_decay_gives_cut_decay}
  Let $S=\begin{bmatrix} A & B \\ D & C \end{bmatrix}$ satisfy the
  exponential decay property with respect to distance between vertices
  \begin{equation}
    |S_{ij}| \leq ce^{-\alpha d(i,j)} 
  \end{equation}
  for all $i,j=1,\dots,n$ with positive constants $c$ and $\alpha$ and
  assume that $A=0$ and $C=0$.  Then $S$ also satisfies the
  exponential decay property with respect to distance to cut
  \begin{equation} 
    |S_{ij}| \leq ce^{-\alpha \max(d^\mathrm{cut}(i),d^\mathrm{cut}(j))}
  \end{equation}
   for all $i,j=1,\dots,n$, where $d^\mathrm{cut}(i)$ is defined as in~\eqref{eq:cut_dist_single_matrix}.
\end{lemma}
\begin{proof}
  For all $i \in I_A$, $j \in I_C$ we have that
  \begin{equation}
    d(i,j) \geq \max (d^\mathrm{cut}(i), d^\mathrm{cut}(j))
  \end{equation}
  and therefore 
  \begin{equation}
     |S_{ij}| \leq ce^{-\alpha d(i,j)} \leq ce^{-\alpha \max
      (d^\mathrm{cut}(i), d^\mathrm{cut}(j))}.
  \end{equation}
  The same bounds clearly hold also for $i \in I_C$, $j \in I_A$ and
  since all other entries are zero, the exponential decay property
  with respect to distance to cut is thus satisfied.
\end{proof}

\begin{proof}[Proof of Theorem~\ref{thm:iter_refine_preserves_decay}]
  Since $\kappa_n$ is uniformly bounded,
  by~\eqref{eq:iter_refine_no_of_iters} the number of iterations until
  the stopping criterion is triggered is also uniformly bounded. All
  matrices in \eqref{eq:alg_matrices_first} through
  \eqref{eq:alg_matrices_last} are therefore the result of a bounded
  number of matrix-matrix multiplications and additions. Repeated
  application of Theorem~\ref{thm:exp_decay_product} implies the
  exponential decay property \eqref{eq:alg_decay_1} for each matrix in
  \eqref{eq:alg_matrices_first} through \eqref{eq:alg_matrices_last}.

  By Lemma~\ref{lem:exp_decay_gives_cut_decay}, the matrix
  $\begin{bmatrix} 0 & B_n \\ B_n^* & 0 \end{bmatrix}$ satisfies an
  exponential decay property with respect to distance to the
  cut. Therefore and by Corollary~\ref{cor:cut_decay_product} the
  matrix
  \begin{equation}\label{eq:delta_part}
    \begin{bmatrix} 0 & B_n \\ B_n^* & 0 \end{bmatrix} Q_n 
  \end{equation}
  also satisfies an
  exponential decay property with respect to distance to the
  cut. Corollary~\ref{cor:cut_decay_product} also applies to the
  matrix $\delta_0$ since it is the result of multiplication of
  $Q_n^*$ with \eqref{eq:delta_part}. In fact,
  Corollary~\ref{cor:cut_decay_product} may be applied to each matrix
  in \eqref{eq:alg_matrices_second} through
  \eqref{eq:alg_matrices_last} since a matrix with exponential decay
  away from the cut is involved in each product.
\end{proof}

Theorem~\ref{thm:on_entries_with_expdecay} and
Theorem~\ref{thm:iter_refine_preserves_decay} imply that for any
$\varepsilon > 0$ each matrix in \eqref{eq:alg_matrices_first} through
\eqref{eq:alg_matrices_last} contains at most $O(n)$ entries greater
than $\varepsilon$ in magnitude. Furthermore, the number of entries
greater than $\varepsilon$ in magnitude in each column or row is
bounded by a constant independent of $n$.  Let $p(n)$ be a function
such that the cut size increases as $O(p(n))$ with system size $n$.
Then, Theorem~\ref{thm:k_entries_with_cutdecay} and
Theorem~\ref{thm:iter_refine_preserves_decay} imply that for any
$\varepsilon > 0$ each matrix in \eqref{eq:alg_matrices_second}
through \eqref{eq:alg_matrices_last} contains at most $O(p(n))$
entries greater than $\varepsilon$ in magnitude.

\section{Recursive and localized inverse factorization}\label{sec:recursive}
Associate with $S$ a binary principal submatrix tree corresponding to
a recursive binary partition of the index set $\{1,\dots,n\}$ of
$S$. This recursive partitioning may continue down to single matrix
elements or stop at some higher level. 

Given such a binary principal submatrix tree, one can use
Algorithm~\ref{alg:iter_refine} in a recursive construction of an
inverse factor. The matrix is split into four quadrants according to
the binary partition of the index set, the method is called
recursively on the two diagonal submatrices, and then the iterative
refinement of Algorithm~\ref{alg:iter_refine} is used to obtain the
inverse factor of the whole matrix. Our recursive inverse
factorization algorithm is given by Algorithm~\ref{alg:rec-inv-fact}.
Algorithm~\ref{alg:rec-inv-fact} was first proposed
in~\cite{rubenssonBockHolmstromNiklasson}, but includes here the new
stopping criterion for iterative refinement presented in
Section~\ref{sec:stopping}. On line~\ref{algline:recinvfact_stopcrit} either one of the spectral 
and Frobenius matrix norms may be used.  

Our localized inverse factorization, given by
Algorithm~\ref{alg:loc-inv-fact}, is obtained in essentially the same
way, but makes use of the localized construction of starting guess and
the localized iterative refinement of Algorithm~\ref{alg:local_refine}.

Theorem~\ref{thm:recinvfact} implies convergence of the iterative
refinement for each level of the recursion in
Algorithm~\ref{alg:rec-inv-fact} and~\ref{alg:loc-inv-fact}. Furthermore, it follows
from~\eqref{eq:thm_rif_1} and~\eqref{eq:thm_rif_2} that the bound of
the number of iterations, needed to reach an accuracy $\|\delta_i\|_2 <
\varepsilon$, given by \eqref{eq:iter_refine_no_of_iters}, holds for all
levels in the recursion.

\begin{algorithm}
  \caption{Recursive inverse factorization \label{alg:rec-inv-fact}}
\begin{algorithmic}[1]
  \Procedure{rec-inv-fact}{S}
  \State \textbf{input:} Hermitian positive definite matrix $S$ with an associated binary principal submatrix tree.
  \If{lowest level}
  \State Factorize $S^{-1} = ZZ^*$ and
  \Return Z
  \EndIf
  \State With the given binary principal submatrix partition $S = \begin{bmatrix} A & B \\ B^* & C \end{bmatrix}$, \label{algline:submat-partition}
  \State compute $Z_A = $ {\sc rec-inv-fact}(A) and $Z_C = $ {\sc rec-inv-fact}(C) \label{algline:recursive_calls}
  \State $Z_0 = \begin{bmatrix} Z_A & 0 \\ 0 & Z_C \end{bmatrix}$
  \State $\delta_{0} = I-Z_{0}^*SZ_{0}$
  \Repeat \textbf{ for} $i=0,1,\dots$
  \State $Z_{i+1} = Z_i \sum_{k=0}^m b_k\delta_i^k$
  \State $\delta_{i+1} = I-Z_{i+1}^*SZ_{i+1}$ \label{algline:regular_delta_computation}
  \Until{$\|\delta_{i+1}\| > \|\delta_{i}\|^{m+1}$}\label{algline:recinvfact_stopcrit}
  \State \Return $Z_{i+1}$
  \EndProcedure
\end{algorithmic}
\end{algorithm}

\begin{algorithm}
  \caption{Localized inverse factorization \label{alg:loc-inv-fact}}
\begin{algorithmic}[1]
  \Procedure{loc-inv-fact}{S}
  \State \textbf{input:} Hermitian positive definite matrix $S$ with an associated binary principal submatrix tree.
  \If{lowest level}
  \State Factorize $S^{-1} = ZZ^*$ and
  \Return Z
  \EndIf
  \State With the given binary principal submatrix partition $S = \begin{bmatrix} A & B \\ B^* & C \end{bmatrix}$, 
  \State compute $Z_A = $ {\sc loc-inv-fact}(A) and $Z_C = $ {\sc loc-inv-fact}(C)
  \State $Z_0 = \begin{bmatrix} Z_A & 0 \\ 0 & Z_C \end{bmatrix}$
  \State $\delta_{0} =
  -\begin{bmatrix} 0 & Z_A^*BZ_C \\ Z_C^*B^*Z_A & 0 \end{bmatrix}$
  \Repeat \textbf{ for} $i=0,1,\dots$
  \State $M_i  =  Z_i (\sum_{k=1}^m b_k\delta_i^k)$
  \State $Z_{i+1}  =  Z_i + M_i$
  \State $\delta_{i+1} = \delta_i - Z_{i+1}^*(SM_i) - (M_i^*S)Z_i$
  \Until{$\|\delta_{i+1}\| > \|\delta_{i}\|^{m+1}$} \label{algline:locinvfact_stopcrit}
  \State \Return $Z_{i+1}$
  \EndProcedure
\end{algorithmic}
\end{algorithm}

\section{Numerical experiments}\label{sec:numerexp}
In this section the localization properties of the localized inverse
factorization, i.e.~Algorithm~\ref{alg:loc-inv-fact}, are
demonstrated. In all experiments, the recursion in
Algorithm~\ref{alg:loc-inv-fact} is continued all the way down to
single matrix elements where $Z = 1/\sqrt{S}$. In this
way the final inverse factor is, up to differences due to floating
point rounding, uniquely determined by
Algorithm~\ref{alg:loc-inv-fact} and the recursive partition of the
index set.
In all numerical experiments $m=1$ was used in the iterative
refinement and the Frobenius norm was used in the stopping criterion
on line~\ref{algline:locinvfact_stopcrit} of the algorithm.

Note that from a computational point of view it would likely be
beneficial to stop the recursion at some predetermined larger block
size and use for example the AINV algorithm~\cite{benzi-ainv} or one
of its variants~\cite{benzi-ainvstable, benzi-ainvblock,
  hierarchic_2007} to compute the inverse factor at the lowest
level. The recursion may for example be stopped when there is no
longer enough sparsity to take advantage of localization in $S$ and/or
when the inverse factorization at that level will run serially,
e.g.\ due to limited resources, so that the parallel features of the
recursive algorithm will anyway not be utilized.

\subsection{Simple lattices}
Our first set of benchmark systems are chosen to clearly demonstrate
the localization behavior for one-, two, and
three-dimensional systems. The systems are also such that the results
should be relatively easy to reproduce, not relying on auxiliary
information, requiring extensive programming effort nor access to a
supercomputer.
We consider adjacency matrices corresponding to one-, two-, and three-dimensional integer lattices, i.e.\ a grid with unit spacing between nearest neighbors.
Diagonal matrix entries are set to $\alpha$ and matrix entries
corresponding to edges between nearest neighbors on the lattice are
set to $\beta$. In the one-dimensional case this gives a tridiagonal
matrix. In the two- and three-dimensional cases the vertices of the
lattice are ordered using a recursive binary divide space
procedure. At each level of the recursion the vertices are sorted
along the greatest dimension and split in two subsets.  Unless
otherwise stated, we use the set of parameters in
Table~\ref{tab:lattice_params}.
\begin{table}
  \begin{center}
    \begin{tabular}{c|c|c|c}
      Lattice & No. of vertices & $\alpha$ & $\beta$ \\
      \hline
      \includegraphics[width=0.07\textwidth]{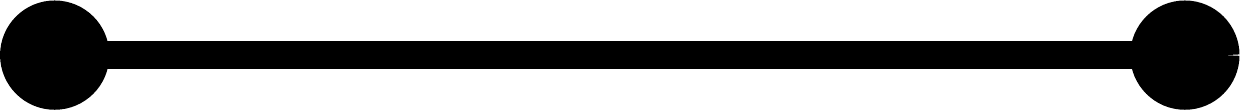} & 512 & 1 & 0.25 \\
      \includegraphics[width=0.07\textwidth]{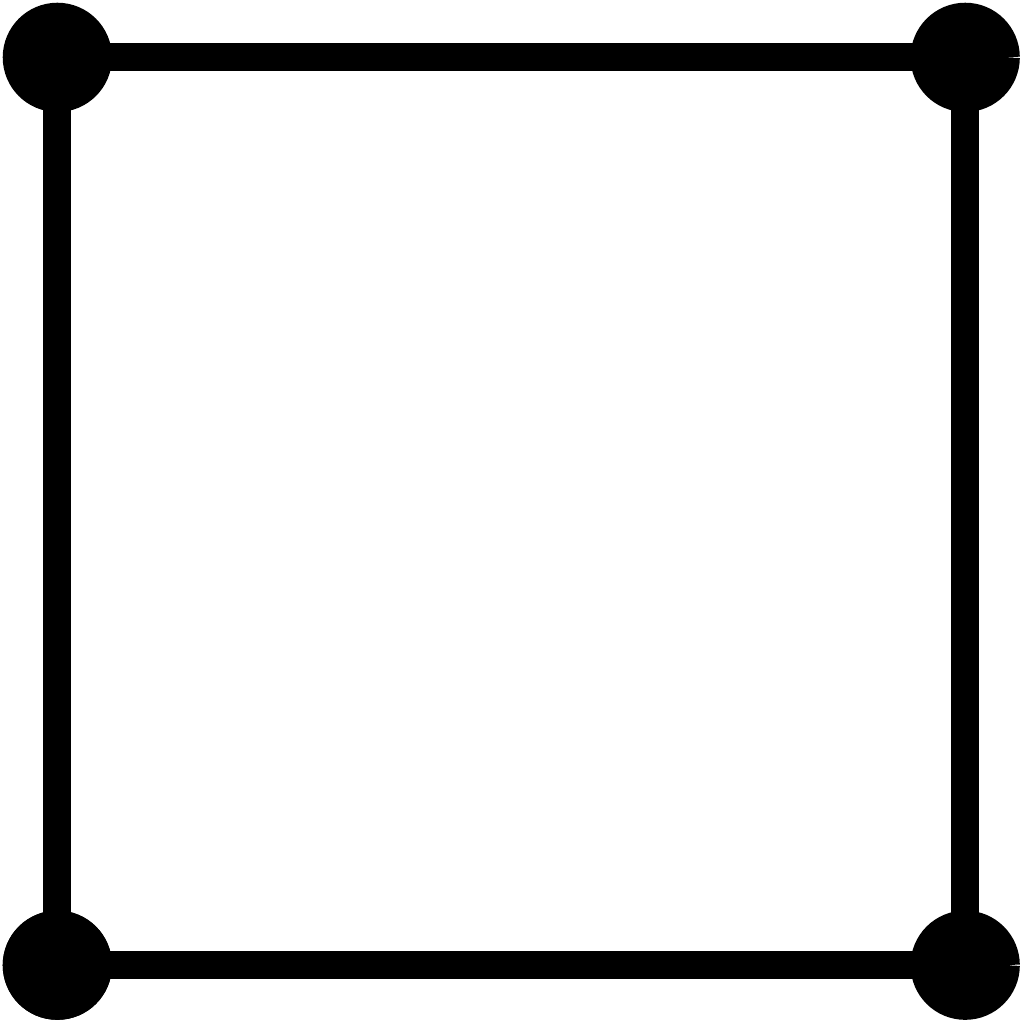} & $64\times 64$ & 1 & 0.05 \\
      \includegraphics[width=0.07\textwidth]{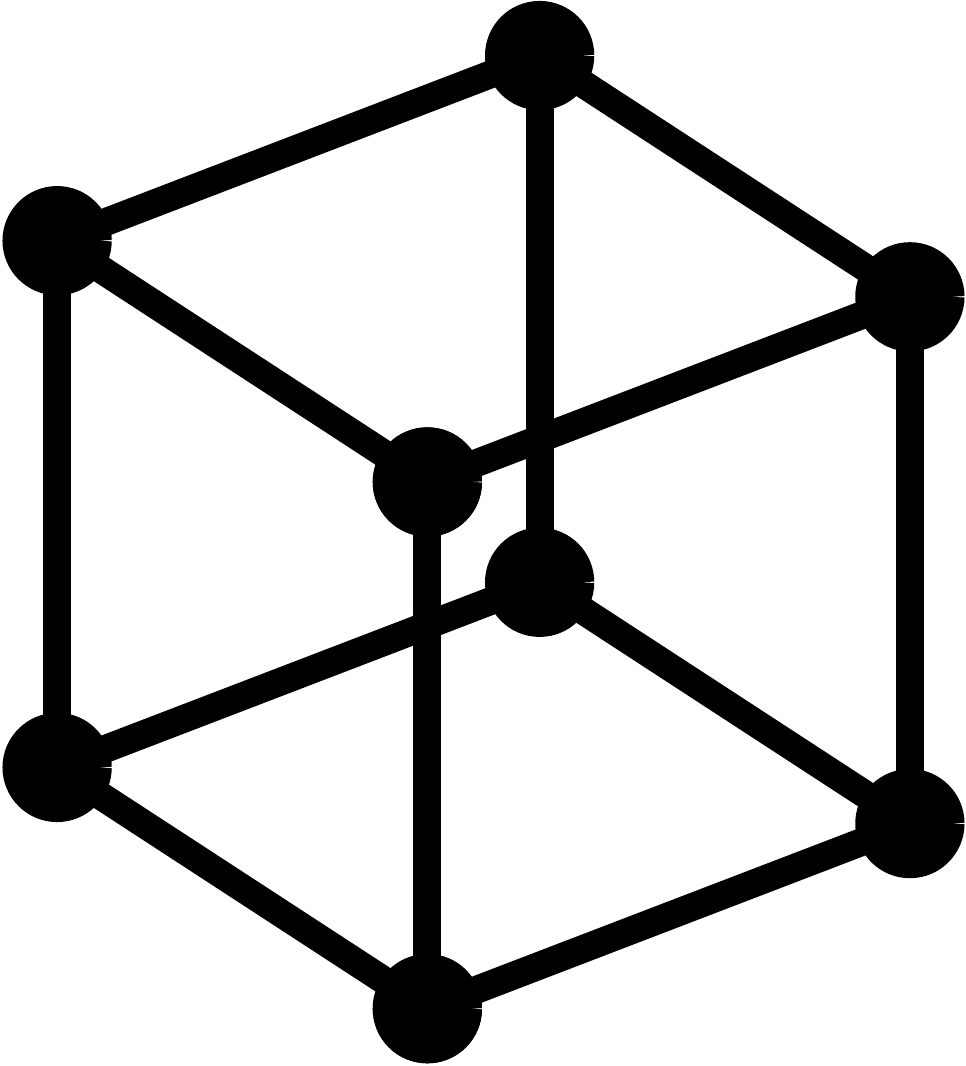} & $16\times 16\times 16$ & 1 & 0.01  \\
      \hline
    \end{tabular}
    \caption{Parameters used to set up the lattice benchmark
      systems.\label{tab:lattice_params}}
  \end{center}
\end{table}

\subsubsection{Localization in inverse factor}
Figure~\ref{fig:localization_Z} demonstrates the localization in the
final inverse factor $Z$ produced by Algorithm~\ref{alg:loc-inv-fact}
for the 1D, 2D, and 3D lattice systems.  Note that, in the upper
panels showing images of the $Z$-matrices, a single blue color is used
to indicate matrix elements with absolute value smaller than
$10^{-30}$.  Although not visible in the figures, the decay of course
continues below $10^{-30}$.

In the image of $Z$ in the upper left panel for the one-dimensional
system the matrix element magnitude clearly decays away from the
diagonal which in this case corresponds to decay with distance between
vertices. This decay with distance is more difficult to grasp looking
only at the images in the upper panels for the two- and
three-dimensional systems. However, the lower panels reveals an even
faster decay with distance for the two- and three-dimensional systems,
which is due to the smaller matrix entries corresponding to edges
between nearest neighbors determined by the parameter $\beta$, see
Table~\ref{tab:lattice_params}.

\begin{figure}
  \begin{center}
    \includegraphics[width=\textwidth]{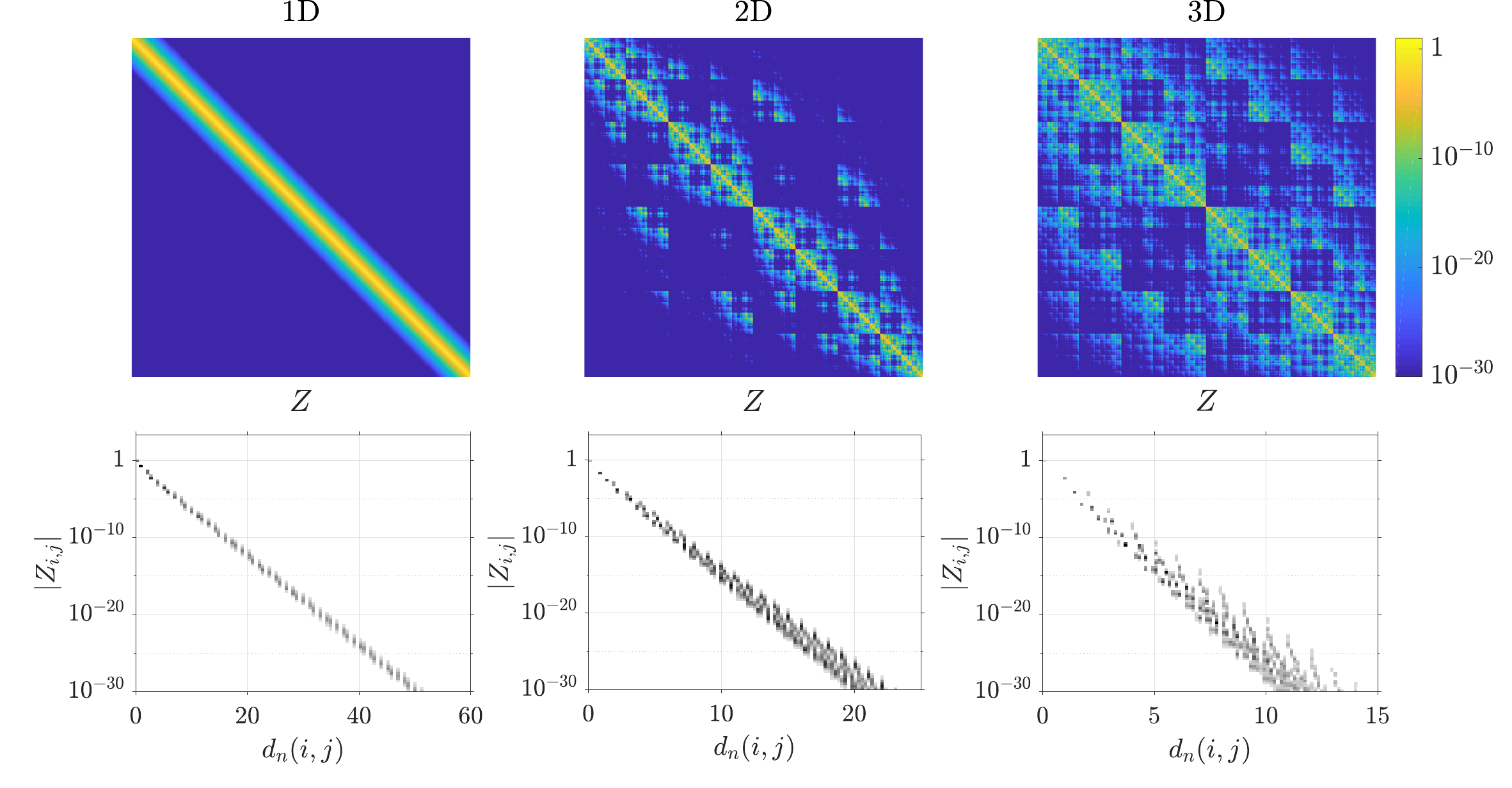}
  \end{center}
  \caption{Localization in $Z$ produced by
    Algorithm~\ref{alg:loc-inv-fact} for the lattice 1D, 2D, and 3D
    benchmark systems, with parameters and dimensions given in
    Table~\ref{tab:lattice_params}.  The upper panels show the matrix
    $Z$ as an image where the color indicates the magnitude of the
    matrix elements. The lower panels show the decay of matrix
    elements with distance between vertices on the lattice,
    corresponding to each of the upper panels.
    \label{fig:localization_Z}
  }
\end{figure}

\subsubsection{Localization in correction matrices}
The complete recursion in Algorithm~\ref{alg:loc-inv-fact} can be seen
as a sequence of corrections to the inverse factor, i.e.\ $Z =
\sum_{l=0}^{r-1} K_l$ where the subscript denotes the level of
recursion with $K_0$ being the correction at the root level of the
recursion and $r$ is the number of recursion levels. 
The correction matrix $K_l$ is a block diagonal matrix, where each
block corresponds to a node of the binary tree of recursive calls at
level $l$ and is, for $l<r-1$, equal to the sum of the matrices $M_i$ in the
iterative refinement corresponding to this node.
The matrix $K_{r-1}$ is, in our case where
we continue the recursion down to single matrix elements, the diagonal
matrix given by $K_{r-1} = \rm{diag}(S)^{-1/2}$.
Note that for $l>0$, the correction matrix $K_l$ 
spans all branches in the binary tree of recursion at level $l$.
We will use the correction matrices as representatives for the
matrices in Algorithm~\ref{alg:loc-inv-fact} corresponding to the
matrices in \eqref{eq:alg_matrices_second} through
\eqref{eq:alg_matrices_last}, which all show similar localization
behavior.

Figure~\ref{fig:localization_KX} shows images of the correction
matrices $K_l$ for $l=0,1,2,3$, for the one-, two-, and
three-dimensional lattice benchmark systems. It may be noted that
these figures correspond directly to the images of $Z$ in
Figure~\ref{fig:localization_Z}. One may for example note that the
upper right and lower left quadrants of $K_0$ are identical to the
corresponding submatrices of $Z$. This is expected as correction to
those submatrices only occurs at the root level of the
recursion. Again the localization, with rapid decay away from the cut
(or cuts for $l=1,2,3$), is clearly seen in the one-dimensional case,
but not as easily seen in the two- and three-dimensional cases.

To more clearly see the localization behavior
we plot in Figure~\ref{fig:localization_K0} the decay with distance
between vertices and the decay away from the cut for the root level
correction matrices $K_0$. Clearly, there is rapid decay both with
distance between vertices and away from the cut, which may be compared
to the decay in the matrices $Z$ seen in the lower panels in
Figure~\ref{fig:localization_Z}.

The key advantage of the localized inverse factorization of
Algorithm~\ref{alg:loc-inv-fact} over the regular recursive inverse
factorization of Algorithm~\ref{alg:rec-inv-fact} is its superior
localization properties. Algorithm~\ref{alg:rec-inv-fact} includes
products of matrices that feature localization of the type seen in
Figure~\ref{fig:localization_Z}, i.e.\ exponential decay with distance
between vertices. See for example the construction of $\delta_i$ on
line~\ref{algline:regular_delta_computation} of
Algorithm~\ref{alg:rec-inv-fact} with matrix products involving $Z_i$
and $S$.  Algorithm~\ref{alg:loc-inv-fact}, on the other hand,
involves only operations, matrix products and additions, where at
least one of the matrices feature localization of the type seen in
Figures~\ref{fig:localization_KX} and~\ref{fig:localization_K0}. This
means that, if only matrix elements of significant magnitude are
included in the calculation, the computational cost of
Algorithm~\ref{alg:loc-inv-fact} may be way lower than that of
Algorithm~\ref{alg:rec-inv-fact}.

\begin{figure}
  \begin{center}
    \includegraphics[width=\textwidth]{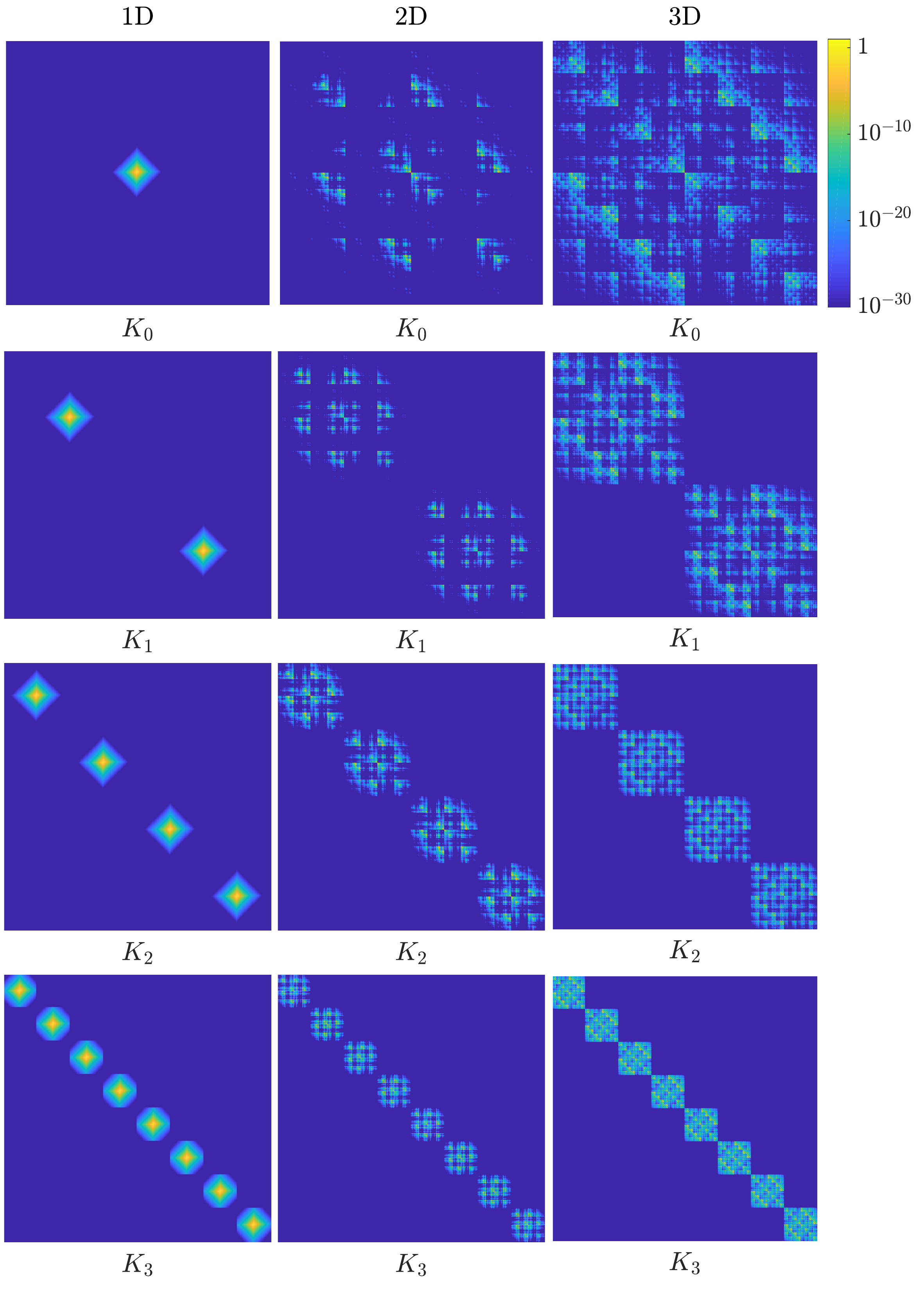}
  \end{center}
  \caption{Correction matrices $K_l,\, l=0,1,2,3$ as images,
    extracted from Algorithm~\ref{alg:loc-inv-fact} for the
    calculations that produced the $Z$-matrices in
    Figure~\ref{fig:localization_Z}. 
    \label{fig:localization_KX}
  }
\end{figure}

\begin{figure}
    \begin{center}
      \includegraphics[width=\textwidth]{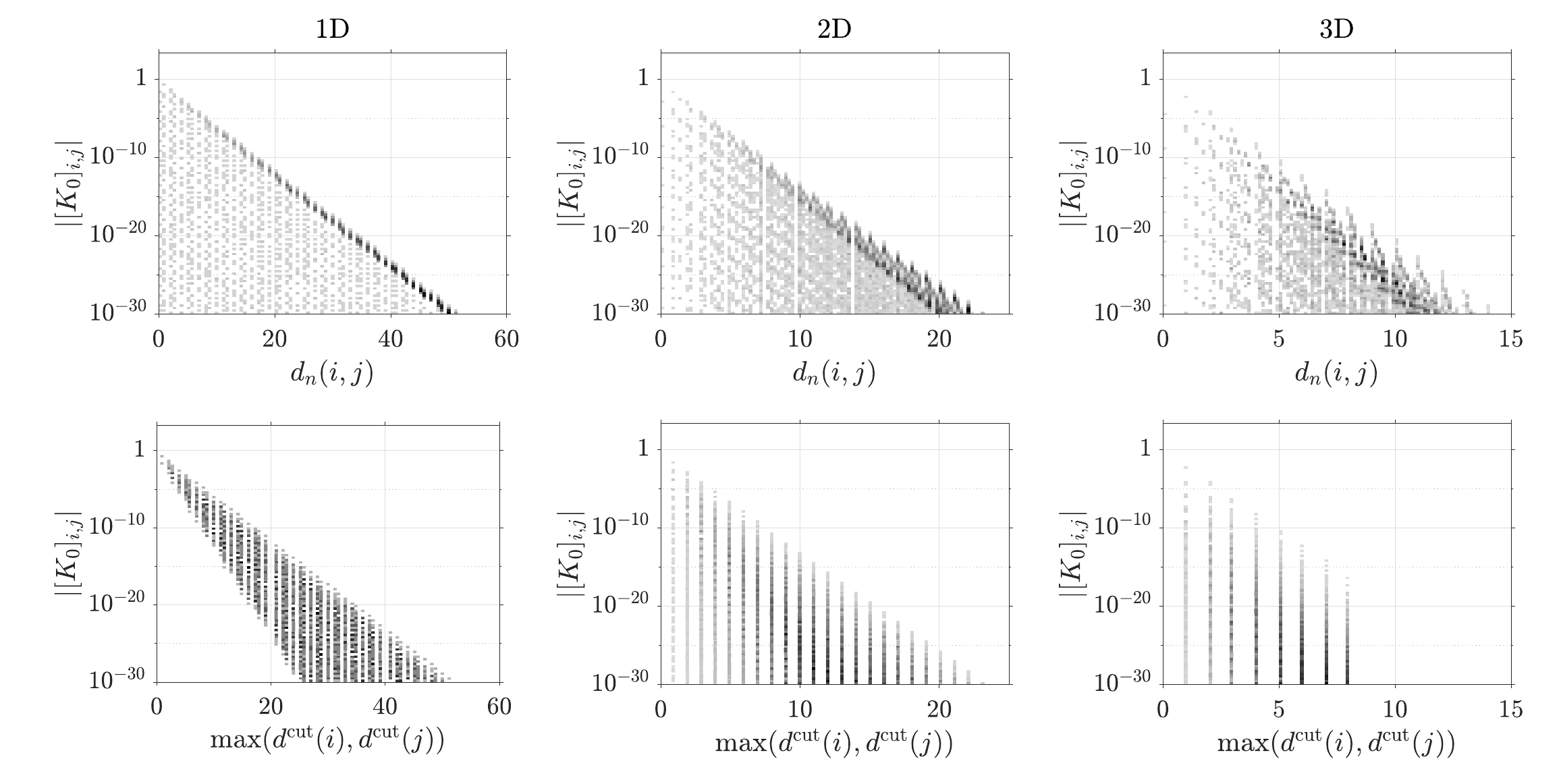}
    \end{center}
  \caption{Localization in the root level correction matrix $K_0$ for
    the calculations that produced the $Z$-matrices in
    Figure~\ref{fig:localization_Z}. The upper panels show the decay
    of matrix elements with distance between vertices on the
    lattice. The lower panels show the decay of matrix elements with
    distance away from the cut.
    \label{fig:localization_K0}
  }
\end{figure}

\subsubsection{Localization with increasing system size}
So far all calculations have been for the fixed system sizes given in
Table~\ref{tab:lattice_params}. We are now interested in how the
localization changes with increasing system size. We consider again
adjacency matrices corresponding to integer lattices but with
increasing dimension. We set the parameters $\alpha$ and $\beta$ as in
Table~\ref{tab:lattice_params}.  We consider one-dimensional grids
with 8, 16, 32, 64, 128, 256, and 512 vertices, two-dimensional grids
with $s\times s$ vertices where $s = 4, 8, 16, 32, 64$, and
three-dimensional grids with $s\times s \times s$ vertices where $s =
2, 4, 8, 16$.

To see how the localization changes with system size we plot in
Figure~\ref{fig:localization_system_size} the number of matrix entries
in $Z$ and $K_0$ whose absolute value is above a given threshold value
with increasing system size. For the final inverse factor $Z$ the
number of significant matrix entries increases linearly with system
size which is consistent with an exponential decay with respect to
distance between vertices $|Z_{ij}| \leq ce^{-\alpha
  d_n(i,j)},\ i,j=1,\dots,n$ with $c$ and $\alpha$ independent of $n$,
see Theorem~\ref{thm:on_entries_with_expdecay}.

For the correction matrices $K_0$
Figure~\ref{fig:localization_system_size} displays very different
behavior for the one-, two-, and three-dimensional systems. This can
be understood by considering how the cut size varies with system size.
In the 1D case, the number of vertices within any fixed distance $R$
from the cut, i.e.\ \emph{the cut size}, is bounded by a constant
independent of $n$.  In the 2D case, the cut size increases as
$O(\sqrt{n})$ and in the 3D case, the cut size increases as
$O(n^{2/3})$.  We recall from
Theorem~\ref{thm:k_entries_with_cutdecay} that a matrix that satisfies
both exponential decay with distance between vertices and away from a
cut with cut size $p(n)$ contains at most $O(p(n))$ entries with
magnitude greater than any given constant $\varepsilon$.  The nearly
perfect least squares fits in the lower panels of
Figure~\ref{fig:localization_system_size} indicate that the number of
entries in $K_0$ increases as $O(1)$, $O(\sqrt{n})$, and $O(n^{2/3})$
for the 1D, 2D, and 3D cases respectively.
The results in the lower panels of
Figure~\ref{fig:localization_system_size} are therefore consistent
with $K_0$ satisfying both exponential decay with respect to distance
between vertices $|(K_0)_{ij}| \leq ce^{-\alpha d_n(i,j)}$ and away
from the cut $|(K_0)_{ij}| \leq ce^{-\alpha
  \max(d^\mathrm{cut}_n(i),d^\mathrm{cut}_n(j))}$ with $c$ and
$\alpha$ independent of $n$.

\begin{figure}
  \begin{center}
    \includegraphics[width=\textwidth]{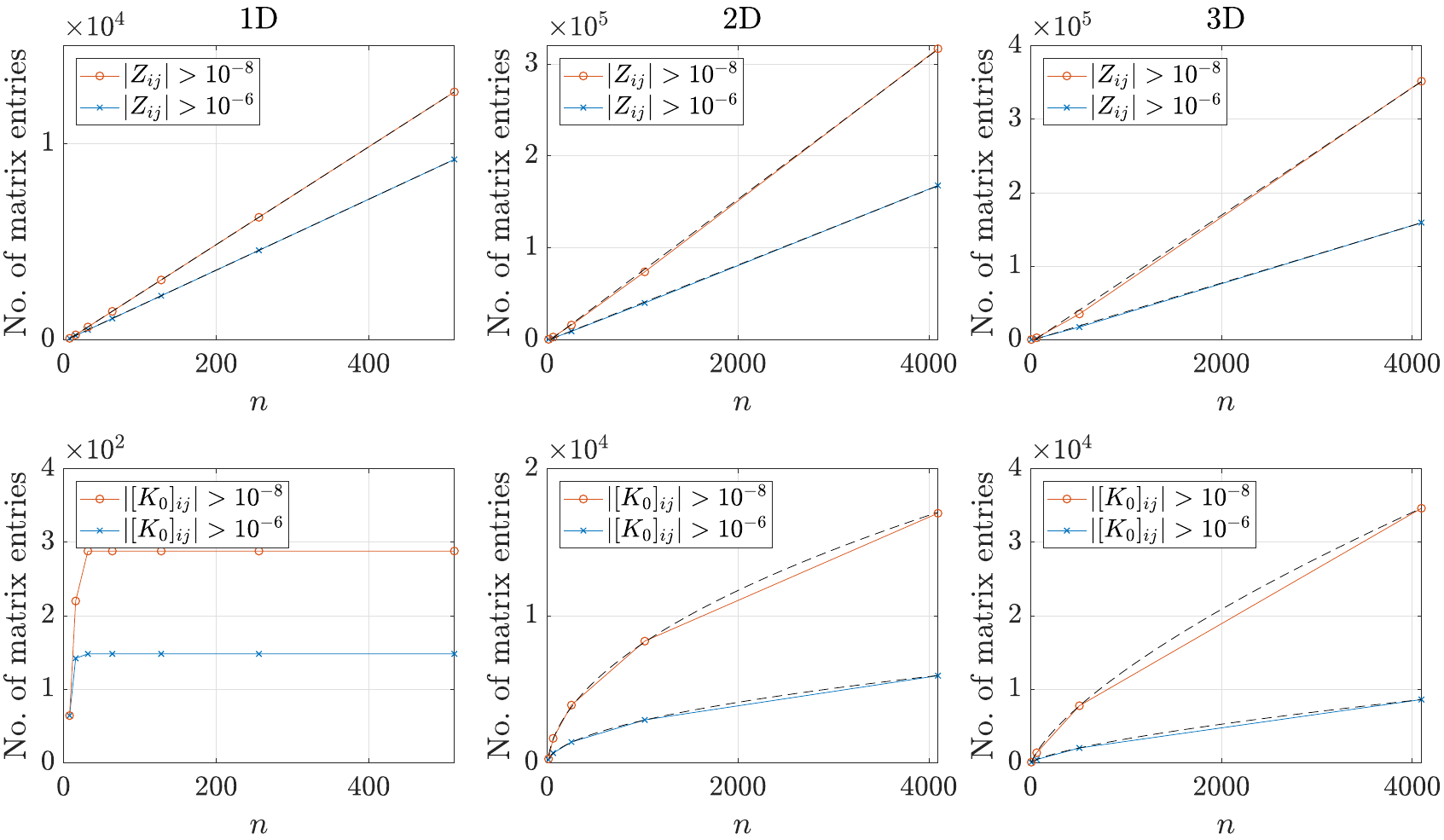}
  \end{center}
  \caption{Number of non-negligible matrix entries in the inverse
    factor $Z$ and the root level correction matrix $K_0$ as a
    function of system size, for Algorithm~\ref{alg:loc-inv-fact} and
    systems based on one-, two-, and three-dimensional lattices. The
    upper panels show the number of entries in $Z$ above $10^{-6}$ and
    $10^{-8}$. The dashed help lines in the upper panels show linear
    least squares fits to the data. The lower panels show the number
    of entries in $K_0$ above $10^{-6}$ and $10^{-8}$.
    The dashed help lines in the middle and right lower panels show
    $c_0 + c_1\sqrt{n}$ and $c_0 + c_1n^{1/3} + c_2n^{2/3}$ least
    squares fits, respectively. 
    \label{fig:localization_system_size}
  }
\end{figure}

\subsection{Alkane chains and water clusters}
We consider here application of the localized inverse factorization to
basis set overlap matrices occurring in Hartree--Fock and Kohn--Sham
density functional theory electronic structure calculations using
standard Gaussian basis sets. The overlap matrices were generated
using the Ergo open-source program for linear-scaling electronic
structure calculations~\cite{Ergo-SoftwareX-2018}, publicly available
at {\tt ergoscf.org} under the GNU Public License (GPL) v3.  In the
case of such overlap matrices, each vertex corresponds to a basis
function center and we let the distance function $d(i,j)$ be the
Euclidean distance between basis function centers $i$ and $j$.  The
magnitude of matrix entries decays as $|S_{ij}| \leq c e^{-\alpha
  d(i,j)^2}$ which is even faster than exponential
decay~\cite{bringing_about}.  In all calculations the standard
Gaussian basis set STO-3G was used. 
Similarly to the calculations on the lattice systems the basis
functions were ordered using a recursive binary divide space procedure
based on the coordinates of the basis function centers.

We consider overlap matrices for alkane chains and water clusters. We
run calculations with two different sizes for each type of system to
be able to see if the exponential decay properties are retained when
the system size is increased.  The alkane chain xyz coordinates were
generated using the {\verb|generate_alkane.cc|} program included in
the Ergo source code package. The water cluster xyz coordinates,
originally used in~\cite{KSDFT-linmem-2011}, are available for
download at {\tt ergoscf.org}.

Localization results for the alkane chains are plotted in
Figures~\ref{fig:loc_Z_alkane} and~\ref{fig:loc_K0_alkane}.
Figure~\ref{fig:loc_Z_alkane} shows the magnitude of the matrix
entries in $S^{-1}$ and $Z$ as functions of distance between vertices
or basis function centers along with an image of $Z$ indicating the
magnitude of the matrix entries as in previous figures.
Figure~\ref{fig:loc_K0_alkane} shows the magnitude of the matrix
entries in $K_0$ as a function of distance between vertices and as a
function of distance away from the cut along with an image of $K_0$.
The matrix dimensions in the 386 atom and 1538 atom cases are 898 and
3586, respectively. These numbers, in contrast to the 1D lattice
system considered above, are not powers of 2, which explains why
significant matrix entries of $K_0$ in Figure~\ref{fig:loc_K0_alkane}
corresponding to atom centers close to the cut are not located at
the center of the matrix.
The corresponding localization results for the water clusters are
plotted in Figures~\ref{fig:loc_Z_h2o} and~\ref{fig:loc_K0_h2o}.  The
dashed help lines makes it easy to see that the exponential decay
properties are retained for all matrices $S^{-1}$, $Z$, and $K_0$ and
for both types of systems.

\begin{figure}
  \begin{center}
    \includegraphics[width=\textwidth]{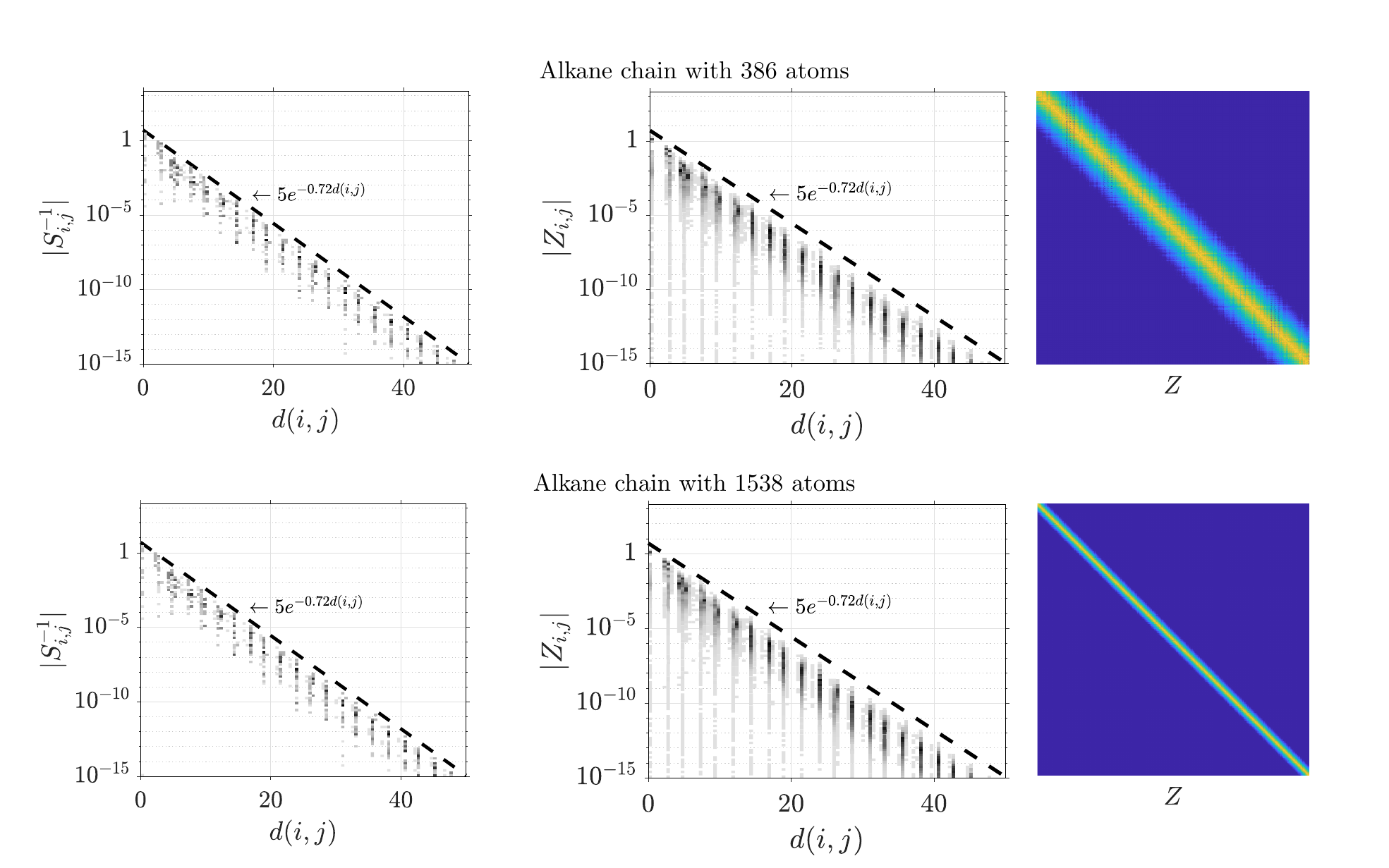}
  \end{center}
  \caption{Left and center panels: Magnitude of matrix entries as a function of
    distance between vertices (basis function centers) for $S^{-1}$
    and $Z$ computed using the localized inverse factorization.
    Right panel: Image of the inverse factor $Z$. Upper panels: Alkane
    chain with 386 atoms. Lower panels: Alkane chain with 1538 atoms.
    \label{fig:loc_Z_alkane}
  }
\end{figure}

\begin{figure}
  \begin{center}
    \includegraphics[width=\textwidth]{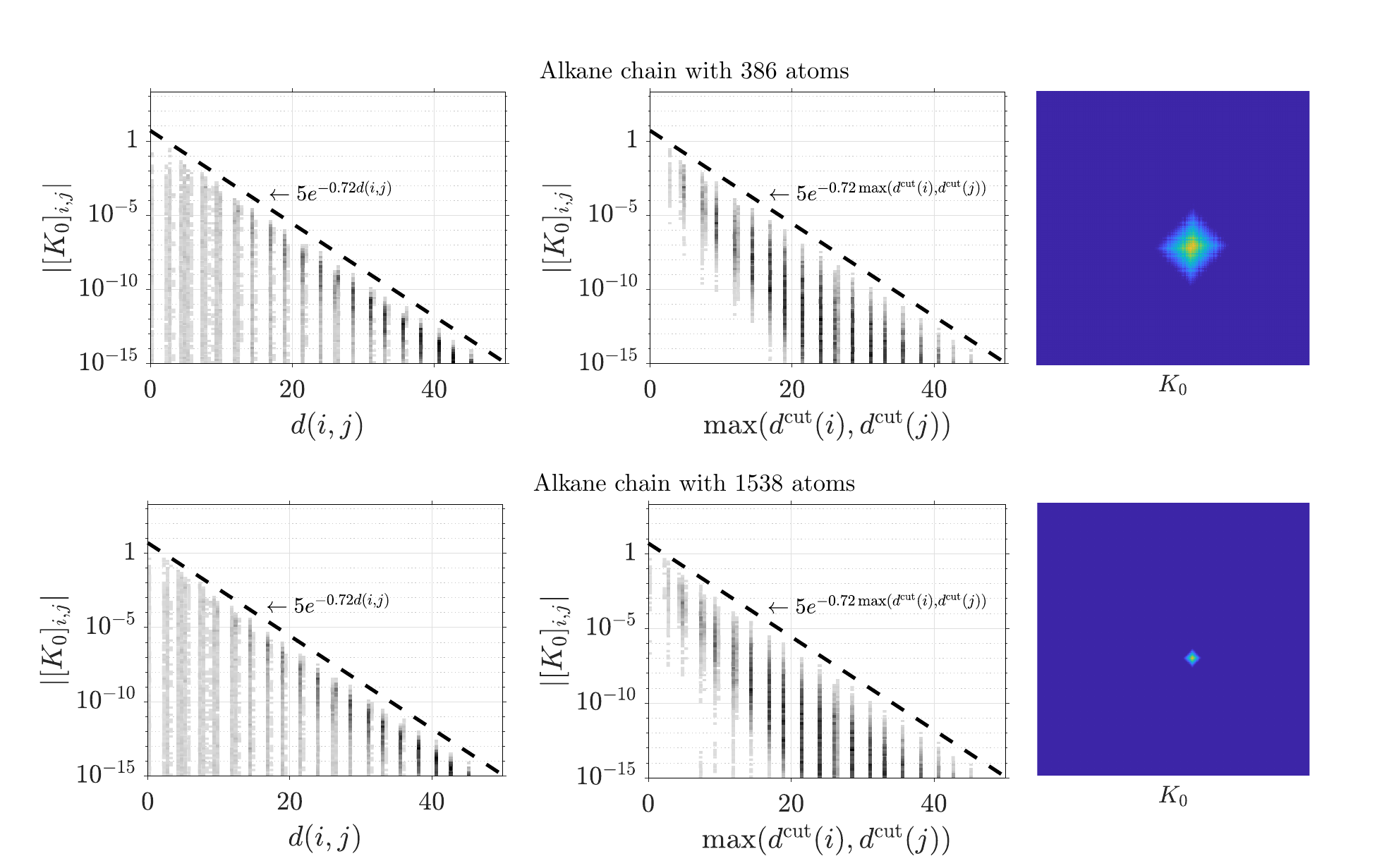}
  \end{center}
  \caption{Left and center panels: Magnitude of matrix entries as a function of
    distance between vertices and as a function of distance from cut
    for $K_{0}$. Right panel: Image of the correction matrix
    $K_0$. Upper panels: Alkane chain with 386 atoms. Lower panels:
    Alkane chain with 1538 atoms.
    \label{fig:loc_K0_alkane}
  }
\end{figure}

\begin{figure}
  \begin{center}
    \includegraphics[width=\textwidth]{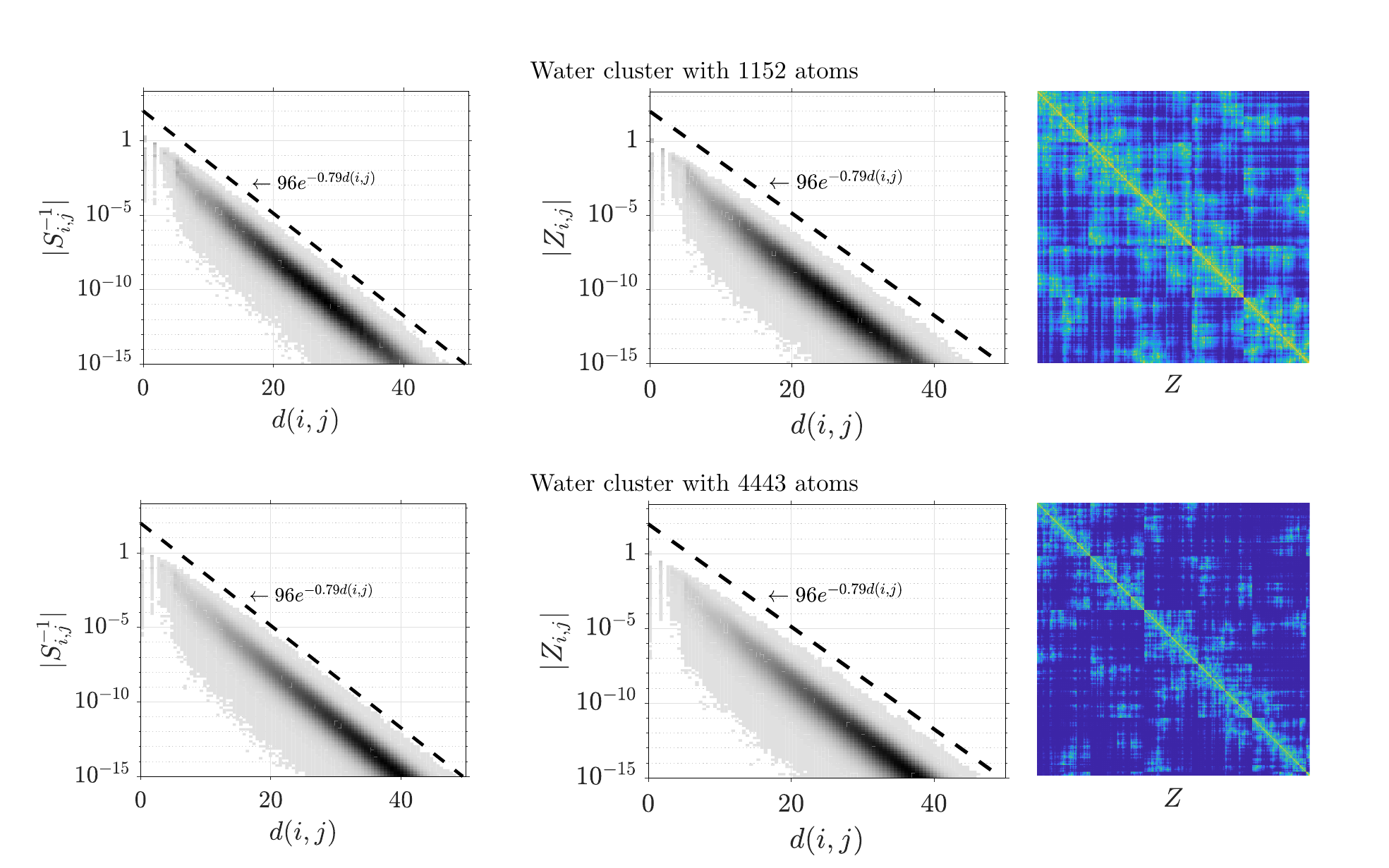}
  \end{center}
  \caption{Left and center panels: Magnitude of matrix entries as a function of
    distance between vertices (basis function centers) for $S^{-1}$
    and $Z$ computed using the localized inverse factorization.  Right
    panel: Image of the inverse factor $Z$. Upper panels: Water
    cluster with 1152 atoms. Lower panels: Water cluster with 4443
    atoms.
    \label{fig:loc_Z_h2o}
  }
\end{figure}

\begin{figure}
  \begin{center}
    \includegraphics[width=\textwidth]{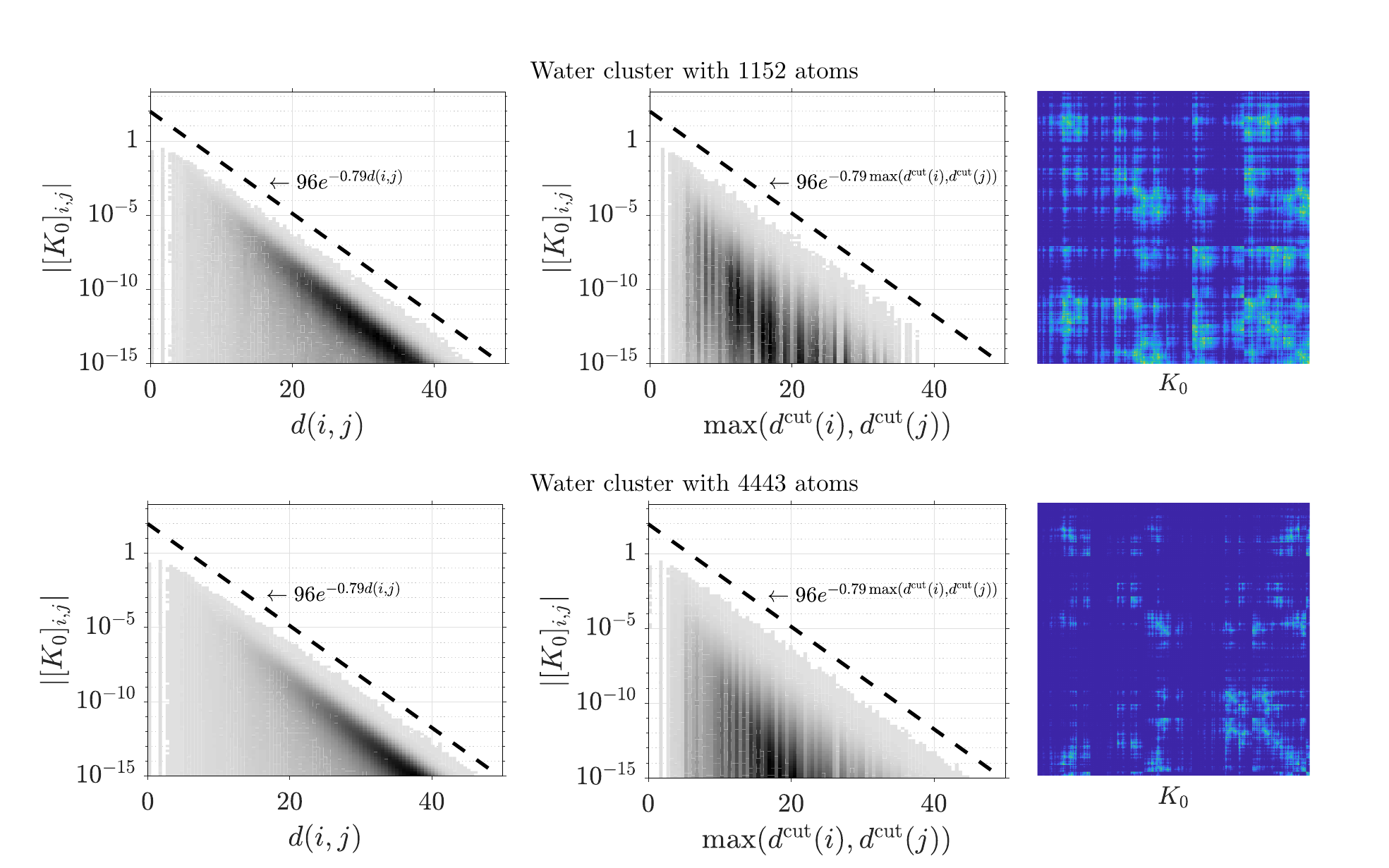}
  \end{center}
  \caption{Left and center panels: Magnitude of matrix entries as a function of
    distance between vertices and as a function of distance from cut
    for $K_{0}$. Right panel: Image of the correction matrix
    $K_0$. Upper panels: Water cluster with 1152 atoms. Lower panels:
    Water cluster with 4443 atoms.
    \label{fig:loc_K0_h2o}
  }
\end{figure}

\section{Concluding remarks}\label{sec:concl}
Previous work on the computation of inverse factors has to large
extent focused on approximations used as preconditioners for iterative
solution of linear systems~\cite{Benzi_1999_305}.  Examples, besides
the AINV algorithms already mentioned, include
FSAI~\cite{Kolotilina1993} and more recent variants~\cite{Franceschini2018}.
Methods making use of a recursive
partitioning of the matrix is used in direct methods for factorization
such as multifrontal
methods~\cite{davis_rajamanickam_sid-lakhdar_2016}. In this case the
matrix is seen as the adjacency matrix of a graph and the matrix is
partitioned using a three-by-three block partition corresponding to a
vertex separator of the graph. Multifrontal methods have typically
been used for example for the Cholesky decomposition and not its inverse
factor. However, the multifrontal approach could also be combined with
the AINV algorithm for direct computation of the inverse
Cholesky factor.  A starting guess for iterative refinement is
in~\cite{Negre2016} built up from the inverse factors of overlapping
principal submatrices.  The convergence has not been proven but this
approach could possibly lead to a lower initial factorization error
and reduced number of iterative refinement iterations compared to our
binary partition at the expense of more computations to construct the
starting guess. Note though that for large enough systems with
localization we have, under certain assumptions, shown that the cost
of our localized inverse factorization is completely dominated by the
solutions of the subproblems and not the iterative refinement used to
glue together their solutions.

A strength of the localized inverse factorization algorithms is the
proved convergence for any binary partition of the index
set. Inappropriate partitions at worst lead to poor localization and
higher computational cost.  We used in our numerical experiments a
straightforward space-dividing algorithm for the recursive binary
partition of the index set. More advanced partitioning algorithms
could take into account the magnitude of the entries in $S$ and
attempt to minimize the cut size and the initial factorization error.

Our localized inverse factorization algorithm could be combined with
some direct inverse factorization method that is efficient for small
dense or semi-sparse systems. 
In the numerical experiments we continued the recursion all the way
down to single matrix elements. It would in general be more efficient
to stop the recursion as soon as the matrix size is smaller than some
predetermined block size and use for example one of the AINV
algorithms to compute the inverse factor.  Another possibility is to
use regular recursive inverse factorization for intermediate levels in
the recursion.

We showed in Section~\ref{sec:localization} that under certain
assumptions both the regular and localized iterative refinement of
Algorithms~\ref{alg:iter_refine} and~\ref{alg:local_refine} involve
only matrices with exponential decay with distance between
vertices. However, all matrices involved in
Algorithm~\ref{alg:local_refine} also have the property of exponential
decay away from the cut. This means that whereas the number of
significant matrix entries in Algorithm~\ref{alg:iter_refine} grows
not faster than $O(n)$, the number of significant matrix entries in
Algorithm~\ref{alg:local_refine} also does not grow faster than the
cut size. For binary partitions with a cut size increasing as $o(n)$
and large enough systems the localized algorithm is therefore superior
to the regular version.
Note that we have not provided a rigorous proof for the
localization features of the full recursive and localized inverse
factorization algorithms demonstrated in Section~\ref{sec:numerexp}.
There are two important differences compared to
Algorithms~\ref{alg:iter_refine} and~\ref{alg:local_refine}. Firstly,
it is difficult to come up with a strict bound for the number of
iterations since the stopping criterion relies on stagnation due to
numerical errors.  Secondly, the computed inverse factor is the result
of a sum over $r = O(\log(n))$ correction matrices, where $r$ is the
number of levels in the recursion. Clearly, $r$ tends to infinity
together with $n$ so that the number of correction matrices in the sum
tends to infinity with increasing $n$, which entails another
difficulty. Nevertheless, all our numerical experiments indicate that
the localization properties of Algorithms~\ref{alg:iter_refine}
and~\ref{alg:local_refine} are inherited by their recursive
counterparts.  Furthermore, our results indicate that the total
computational cost can increase linearly with system size if
negligible matrix entries are not stored nor included in the
calculation.  In summary, the theoretical results of
Section~\ref{sec:localization} together with the numerical experiments
of Section~\ref{sec:numerexp} make us confident that the localized
inverse factorization, especially with regard to localization,
represents a dramatic improvement over the regular recursive inverse
factorization.

The localized inverse factorization is also well suited for
parallelization. The two subproblems of computing inverse factors of
the two principal submatrices at each level of the recursion are
completely disconnected and can thus be solved in parallel without any
communication in between.  For systems with localization, any
communication needed in the iterative refinement step to combine the
subproblem solutions involves only data close to the cut. Using an
appropriate partition of the index set, the cut size is vanishingly
small in proportion to the total problem size for large enough
systems.

This work did not concern computational details regarding how to
select matrix entries for removal or other implementation issues such
as how to best parallelize the method on a particular computer
architecture. The efficient implementation of the localized inverse
factorization algorithm will be considered elsewhere.
 
\bibliographystyle{siamplain}
\bibliography{biblio}
\end{document}